\documentclass[11pt,reqno,a4paper]{amsart}
\usepackage[british]{babel}

\usepackage{multicol}
\usepackage{graphicx}
\usepackage{amscd}
\usepackage{amsmath}
\usepackage{amsfonts}
\usepackage{amssymb}
\usepackage{comment}
\usepackage{color}
\usepackage{pdfsync}
\usepackage{bbm}
\usepackage{dsfont}
\usepackage{graphicx}
\usepackage{hyperref}
\usepackage{enumerate}
\usepackage{mathabx}
\usepackage{soul}
\usepackage{rotating}

\definecolor{trp}{rgb}{1,1,1}

\definecolor{red}{rgb}{1,0,.2}

\newtheorem{theorem}{Theorem}[section]
\theoremstyle{plain}

\newtheorem{definition}[theorem]{Definition}

\newtheorem{lemma}[theorem]{Lemma}

\newtheorem{prop}[theorem]{Proposition}

\numberwithin{equation}{section}

\newcommand{\R}{\mathbb{R}}
\newcommand{\N}{\mathbb{N}}

\newcommand{\pr}[1]{\mathrm{proj}_{e^{ss}(#1)}}
\newcommand{\proj}{\mathrm{proj}}
\newcommand{\lv}{\boldsymbol\lambda}

\newcommand{\xv}{\underline{x}}
\newcommand{\yv}{\underline{y}}
\newcommand{\vv}{\underline{v}}
\newcommand{\ii}{\mathbf{i}}

\newcommand{\jj}{\mathbf{j}}
\newcommand{\tv}{\underline{t}}
\newcommand{\wv}{\underline{w}}
\newcommand{\ldim}{\dim_{\mathrm{Lyap}}}

\newcommand{\tvv}{\mathbf{t}}
\newcommand{\clo}{\overline{\mathcal{O}}}

\newcommand{\Alpha}{\mathcal{A}}

\newcommand{\y}{\underline{\mathbf{y}}}

\newcommand{\e}[2]{e^{#2}_{\mathcal{#1}}}
\newcommand{\iv}{\overline{\imath}}
\newcommand{\jv}{\overline{\jmath}}
\newcommand{\icv}{\overrightarrow{\imath}}

\begin{document}
\title[Dimension maximizing measures]{Dimension maximizing measures for self-affine systems}

\author{Bal\'azs B\'ar\'any}
\address[Bal\'azs B\'ar\'any]{Budapest University of Technology and Economics, BME-MTA Stochastics Research Group, P.O.Box 91, 1521 Budapest, Hungary \&
Mathematics Institute, University of Warwick, Coventry CV4 7AL, UK}
\email{balubsheep@gmail.com}

\author{Micha\l\ Rams}
\address[Micha\l\ Rams]{Institute of Mathematics, Polish Academy of Sciences, ul. S\'niadeckich 8, 00-656 Warszawa, Poland}
\email{rams@impan.pl}

\thanks{The research of Bal\'azs B\'ar\'any was supported by the grants EP/J013560/1 and OTKA K104745. Micha\l\ Rams was supported by National Science Centre grant
2014/13/B/ST1/01033 (Poland).}

\begin{abstract}
In this paper we study the dimension theory of planar self-affine sets satisfying dominated splitting in the linear parts and strong separation condition. The main results of this paper is the existence of dimension maximizing Gibbs measures (K\"aenm\"aki measures). To prove this phenomena, we show that the Ledrappier-Young formula holds for Gibbs measures and we introduce a transversality type condition for the strong-stable directions on the projective space.
\end{abstract}
\date{\today}

\maketitle

\thispagestyle{empty}

\section{Introduction and Statements}

Let $\Alpha:=(A_1, A_2,\dots, A_N)$ be a finite set of contracting, non-singular $2\times2$ matrices, and let $\Phi:=\left\{f_i:\xv\mapsto A_i\xv+\tv_i\right\}_{i=1}^N$ be an {\it iterated function system} (IFS) on the plane with affine mappings, where $\|A_i\|<1$ and  $\underline{t}_i\in\R^2$ for $i=1,\dots,N$. It is a well-known fact that there exists an unique non-empty compact subset $\Lambda$ of $\R^2$ such that
$$
\Lambda=\bigcup_{i=1}^Nf_i(\Lambda).
$$
We call the set $\Lambda$ the {\it attractor} of $\Phi$ or {\it self-affine set}.

Let us denote the Hausdorff dimension of a set $X$ by $\dim_HX$. Moreover, denote by $\underline{\dim}_BX$ and by $\overline{\dim}_BX$ the lower and upper box dimension. If the upper and lower box dimensions coincide then we call the common value the box dimension and denoted by $\dim_BX$. For the definitions and basic properties, we refer to Falconer~\cite{Fb1}.

 The image of the unit ball under the affine mapping $f(\xv)=A\xv+\tv$ is an ellipse. The length of the longer and shorter axes of the ellipse depends only on the matrix $A$, and we call these values the {\it singular values} of $A$. We denote the $i$th singular value of $A$ by $\alpha_i(A)$. More precisely, $\alpha_i(A)$ is the positive square root of the $i$th eigenvalue of $AA^*$, where $A^*$ is the transpose of $A$. We note that in this case, $\alpha_1(A)=\|A\|$ and $\alpha_2(A)=\|A^{-1}\|^{-1}$, where $\|.\|$ is the usual matrix norm induced by the Euclidean norm on $\R^2$. Moreover, $\alpha_1(A)\alpha_2(A)=|\det A|$.

The natural cover of these ellipses play important role in the calculation of the dimension of self-affine sets. The image of the unit ball under an affine mapping can be covered by $1$ ball with radius $\alpha_1(A)$, or can be covered by approximately $\alpha_1(A)/\alpha_2(A)$ balls with radius $\alpha_2(A)$. This leads us to the definition of singular value function. For $s\geq0$ define the {\it singular value function} $\phi^s$ as follows
\begin{equation}\label{esvf}
\phi^s(A):=\left\{\begin{array}{cc}
\alpha_1(A)^s & 0\leq s\leq1 \\
\alpha_1(A)\alpha_2(A)^{s-1} & 1<s\leq2 \\
\left(\alpha_1(A)\alpha_2(A)\right)^{s/2} & s>2.
\end{array}\right.
\end{equation}
Falconer~\cite{F} introduced the {\it subadditive pressure}
\begin{equation}\label{esap}
P_{\mathcal{A}}(s):=\lim_{n\rightarrow\infty}\frac{1}{n}\log\sum_{i_1,\dots,i_n=1}^N\phi^s(A_{i_1}\cdots A_{i_n}).
\end{equation}
The function $P_{\mathcal{A}}:[0,\infty)\mapsto\R$ is continuous, strictly monotone decreasing on $[0,\infty)$, moreover $P_{\mathcal{A}}(0)=\log N$ and $\lim_{s\rightarrow\infty}P_{\mathcal{A}}(s)=-\infty$. Falconer~\cite{F} showed that for the unique root $s_0:=s_0(\mathcal{A})$ of the subadditive pressure function $\overline{\dim}_B\Lambda\leq\min\left\{2,s_0\right\}$ and if $\|A_i\|<1/3$ for every $i=1,\dots,N$ then $$\dim_H\Lambda=\dim_B\Lambda=\min\left\{2,s_0\right\}\text{ for Lebesgue-almost every $\tvv=(\tv_1,\dots,\tv_N)\in\R^{2N}$.}$$  The condition was later weakened to $\|A_i\|<1/2$ by Solomyak, see \cite{S}. We call the value $s_0$ the {\it affinity dimension} of $\Phi$. K\"aenm\"aki~\cite{K} showed that for Lebesgue-almost every $\tvv=(\tv_1,\dots,\tv_N)\in\R^{2N}$ there exists an invariant measure $\nu^K$ supported on $\Lambda$ such that $\dim_H\nu^K=\dim_H\Lambda=\min\left\{2,s_0\right\}$. Under our assumptions: SSC (see below) and dominated splitting (see below, Definition~\ref{ddomsplit}) this measure is image of a Gibbs (Definition~\ref{dKaenmaki}), but in general not image of a Bernoulli.

Other type of 'almost surely' result was unknown previously. The main advantage of this paper is to give an almost everywhere condition on the set of matrices instead of on the set of translation vectors.

In this paper we consider IFSs of affinities which satisfy the {\it strong separation condition} (SSC), i.e.
\begin{equation*}
f_i(\Lambda)\cap f_j(\Lambda)=\emptyset\text{ for every $i\neq j$.}
\end{equation*}
We note that the strong separation condition implies $s_0<2$.

Falconer~\cite{F2} proved that if $\Phi$ satisfies a separation condition (milder than SSC) and the projection of $\Lambda$ in every direction contains an interval then the box dimension of a self-affine set is equal to the affinity dimension. Hueter and Lalley~\cite{HL} gave conditions, which ensure that the Hausdorff and box dimension of a self-affine set equal to the affinity dimension.

In the recent paper of B\'ar\'any~\cite{B}, the result of Hueter and Lalley~\cite{HL} was generalised for self-affine measures. That is, under the same conditions of Hueter and Lalley~\cite{HL} the Hausdorff dimension of any self-affine measure is equal to its Lyapunov dimension. In particular, in \cite{B} the author proved that under slightly more general conditions any self-affine measure is exact dimensional and gave a formula, which connects entropy, Lyapunov exponents and the projection of the measure (Ledrappier-Young formula).

Recently, Falconer and Kempton~\cite{FK} used methods from ergodic theory along with properties of the Furstenberg measure and obtained conditions under which certain classes of plane self-affine sets have Hausdorff and box dimension equal to the affinity dimension. By adapting the conditions of Falconer and Kempton~\cite{FK} and B\'ar\'any~\cite{B} we prove that for "typical" linear parts ($\left\{A_i\right\}_{i=1}^N$) if the SSC holds then the dimension of self-affine set is equal to the affinity dimension. Precisely, let
\begin{equation}\label{emxset}
\mathfrak{M}:=\left\{A\in\R_+^{2\times2}\cup\R_-^{2\times2}:0<\frac{|\det A|}{\vvvert A\vvvert^2}<\frac{1}{2}\text{\ and\ }\|A\|<1\right\},
\end{equation}
where
$$
\vvvert A\vvvert=\min\left\{|a|+|b|,|c|+|d|\right\}\text{ for }A=\left[\begin{matrix}
a & b \\
c & d \\
\end{matrix}\right].
$$
Let us define the following sets
\begin{equation}\label{emxset2}
\mathfrak{N}:=\left\{A\in\mathfrak{M}:\|A^{-1}\|\|A\|^2\leq1\right\}\text{\ and\ }\mathfrak{O}_N:=\left\{\mathcal{A}\in\mathfrak{M}^N:s_0(\mathcal{A})>5/3\right\},
\end{equation}
for every $N\geq2$.

\begin{theorem}\label{tmaindim}
	Let $N\geq2$. For $\mathcal{L}_{4N}$-almost every $\mathcal{A}\in\mathfrak{N}^N\bigcup\mathfrak{O}_N$, if $\tvv=(\tv_1,\dots,\tv_N)\in\R^{2N}$ is chosen such that $\Phi:=\left\{f_i:\xv\mapsto A_i\xv+\tv_i\right\}_{i=1}^N$ satisfies the SSC then there exists a measure $\nu^K$ supported on the attractor $\Lambda$ of $\Phi$ such that
	$$
	\dim_H\nu^K=\dim_H\Lambda=\dim_B\Lambda=s_0(\mathcal{A}).
	$$
\end{theorem}

We call the measure $\nu^K$ the \textit{K\"aenm\"aki measure}.

The authors were recently informed of the result of Rapaport~\cite{Ra} and Morris and Shmerkin~\cite{MS}. By applying the main theorem of Rapaport~\cite{Ra}, one can extend the bound $5/3$ to $3/2$ in \eqref{emxset2}. Morris and Shmerkin~\cite{MS} proved similar statement to Theorem~\ref{tmaindim} under significantly different conditions on the matrices.

To prove Theorem~\ref{tmaindim}, we will need a more detailed study of the dimension of invariant measures. More precisely, we extend the results of \cite{B} for the natural projections of Gibbs measures. Theorem~\ref{tmaindim} is studied in higher generality.

\subsection*{Structure of the paper} After the Preliminaries (Section~\ref{sp}) we introduce the main technical result of the paper, the Ledrappier-Young formula generalised for Gibbs measures (Section~\ref{sLY}). In Section~\ref{ssstc} we introduce the strong-stable transversality condition (Definition~\ref{dsstrans}) and show that under this condition there exists a dimension maximizing Gibbs measure (K\"aenm\"aki measure) almost surely. In the last section we show Theorem~\ref{tmaindim} as a consequence of the previous studies.

\section{Preliminaries}\label{sp}

Let $\Sigma=\left\{1,\dots,N\right\}^{\mathbb{Z}}$ be the symbolic space of two side infinite sequences, $\Sigma^+=\left\{1,\dots,N\right\}^{\mathbb{N}}$ be the set of right side and $\Sigma^-=\left\{1,\dots,N\right\}^{\mathbb{Z}^-}$ be the set of left side infinite words. Denote the left shift operator on $\Sigma$ and $\Sigma^+$ by $\sigma$ and denote the right shift operator on $\Sigma$ and $\Sigma^-$ by $\sigma_-$. Thus, $\sigma$ and $\sigma_-$ are invertible on $\Sigma$ and $\sigma^{-1}=\sigma_-$. For any $\ii\in\Sigma$ (or $\jj\in\Sigma^{\pm})$
$$
[\ii|_m^n]:=\left\{\jj\in\Sigma\text{ (or $\jj\in\Sigma^{\pm}$)}:i_k=j_k\text{ for }m\leq k\leq n\right\}.
$$
For an $\ii=(\dots i_{-2}i_{-1}i_0i_1\dots)\in\Sigma$, denote by $\ii_+=(i_0i_1\dots)$ the right-hand side and by $\ii_-=(\dots i_{-2}i_{-1})$ the left-hand side of $\ii$. To avoid confusion, we write also $\ii_+$ if $\ii_+\in\Sigma^+$ and $\ii_-$ if $\ii_-\in\Sigma^-$.

For any $\ii_+,\jj_+\in\Sigma^+$ let $\ii_+\wedge\jj_+=\min\left\{n\geq0:i_n\neq j_n\right\}$. We define $\ii_-\wedge\jj_-=\min\left\{n-1\geq0:i_{-n}\neq j_{-n}\right\}$ similarly.

Let us denote the set of finite length words by $\Sigma^*=\bigcup_{n=0}^{\infty}\left\{1,\dots,N\right\}^n$, and for every $\iv=(i_1,\dots i_n)\in\Sigma^*$ denote the reversed word by $\icv=(i_n,\dots,i_1)$. Sometimes, we may also write $(\Sigma^-)^*$ for finite length words to emphasize the negative indexes.

If $\Phi:=\left\{f_i(\xv)=A_i\xv+\tv_i\right\}_{i=1}^N$ is an iterated function system on $\R^2$ with affine mappings such that $\|A_i\|<1$ for $i=1,\dots,N$, we define the \textit{natural projection} $\pi_-$ from $\Sigma^-$ to $\Lambda$ in a natural way
\begin{equation}\label{enatproj}
\pi^-(\dots i_{-2}i_{-1})=\lim_{n\rightarrow\infty}f_{i_{-1}}\circ\cdots\circ f_{i_{-n}}(\underline{0}).
\end{equation}

Let $\Alpha:=\left\{A_1, A_2,\dots, A_N\right\}$ be a finite set of non-singular $2\times2$ real matrices. Define a map from $\Sigma$ to $\Alpha$ in a natural way, i.e. $A(\ii):=A_{i_0}$. Let $A^{(n)}(\ii):=A(\sigma^{n-1}\ii)\cdots A(\ii)$ for $\ii\in\Sigma$ and $n\geq1$.

\begin{definition}\label{ddomsplit}
	We say that a set $\mathcal{A}=\left\{A_i\right\}_{i=1}^N$ of matrices satisfies the \textnormal{dominated splitting} if there are constants $C,\beta>0$ such that for every $n\geq1$ and every $i_0,\dots,i_{n-1}\in\left\{1,\dots,N\right\}$
	$$
	\frac{\alpha_1(A_{i_0}\cdots A_{i_{n-1}})}{\alpha_2(A_{i_0}\cdots A_{i_{n-1}})}\geq Ce^{n\beta}.
	$$
\end{definition}

Let $C_+:=\left\{(x,y)\in\R^2\backslash\{(0,0)\}:xy\geq0\right\}$ be the \textit{standard positive cone}. A \textit{cone} is an image of $C_+$ under a linear isomorphism and a \textit{multicone} is a disjoint union of finitely many cones. We say that a multicone $M$ is \textit{backward invariant} w.r.t. $\mathcal{A}$ if $\bigcup_{A\in\mathcal{A}}A^{-1}(M)\subset M^o$, where $M^o$ denotes the interior of $M$.

 For a $2\times 2$ matrix $A$ and a subspace $\theta$ of $\R^2$ we introduce the notation $\|A|\theta\|$, which is the norm of $A$ restricted to the subspace $\theta$, i.e. $\|A|\theta\|=\sup_{\vv\in\theta}\|A\vv\|/\|\vv\|$. Since $\theta$ is one dimensional, we get that for any $\vv\neq\underline{0}\in\theta$, $\|A|\theta\|=\|A\vv\|/\|\vv\|$, which is not true in higher dimension.

\begin{lemma}[\cite{ABY}, \cite{BG},\cite{BR}, \cite{Y}]\label{ldomsplit}
	The set $\Alpha$ of matrices satisfies the dominated splitting then for every $\ii\in\Sigma$ there are two one-dimensional subspaces $e^{ss}(\ii),e^s(\ii)$ of $\R^2$ such that
	\begin{enumerate}
		\item $A_{i_0}e^j(\ii)=e^j(\sigma\ii)$ for every $\ii\in\Sigma$ and $j=s,ss$,
		\item\label{ldomsplit2} there is a constant $C>0$ such that for every $n\geq1$ and $\ii\in\Sigma$
		\begin{eqnarray*}
			&& C^{-1}\|A^{(n)}(\ii)|e^s(\ii)\|\leq\alpha_1(A^{(n)}(\ii))\leq C\|A^{(n)}(\ii)|e^s(\ii)\|\text{ and }\\
			&& C^{-1}\|A^{(n)}(\ii)|e^{ss}(\ii)\|\leq\alpha_2(A^{(n)}(\ii))\leq C\|A^{(n)}(\ii)|e^{ss}(\ii)\|,
		\end{eqnarray*}
		\item there is a backward-invariant multicone $M$ that	$$e^s(\ii)=\bigcap_{n=1}^{\infty}A_{i_{-1}}\cdots A_{i_{-n}}(\overline{M^c})\text{ and }e^{ss}(\ii)=\bigcap_{n=1}^{\infty}A_{i_{0}}^{-1}\cdots A_{i_{n-1}}^{-1}(M),$$
		where $\overline{M^c}$ denotes the closure of the complement of $M$.
		\item\label{lpunif} The angle between $e^s(\ii),e^{ss}(\ii)$ is uniformly bounded below.
	\end{enumerate}
	We call the family of subspaces $e^{s}(\ii)$ \textnormal{stable directions} and $e^{ss}(\ii)$ \textnormal{strong stable directions}.
\end{lemma}

Let us observe that $e^s(\ii)$ depends only on $\ii_-$ and $e^{ss}(\ii_+)$ depends only on $\ii_+$, so $e^{ss}$ can be considered as a natural projection from $\Sigma^+$ to $\mathbf{P}^1$, where $\mathbf{P}^1$ denotes the projective space. In particular, $\|A^{(n)}(\ii)|e^s(\ii)\|$ and $\|A^{(n)}(\ii)|e^{ss}(\ii)\|$ describe the local growth in the stable/strong stable directions, and can be considered as finite time approximations of the corresponding Lyapunov exponent.

For $x,y\in\mathbf{P}^1$ denote by $\sphericalangle(x,y)$ the usual metric on $\mathbf{P}^1$, that is the angle between the subspaces corresponding to $x$ and $y$. Thus, Lemma~\ref{ldomsplit}\eqref{lpunif} can be formalized as follows, there exists a constant $C>0$ such that for every $\ii_-\in\Sigma^-$ and $\jj_+\in\Sigma^+$, $\sphericalangle(e^{ss}(\jj_+),e^s(\ii_-))>C$. In the later analysis, the dimension of strong stable directions in $\mathbf{P}^1$ plays an important role.

For any $\vv,\underline{w}\in\R^2$ denote by $\mathrm{Area}(\vv,\underline{w})$ the area of parallelogram formed by $\vv,\underline{w}$.

\begin{lemma}\label{ldistcomp}
	For every $x,y\in\mathbf{P}^1$
	$$
	\frac{\mathrm{Area}(\vv,\underline{w})}{\|\vv\|\|\underline{w}\|}\leq\sphericalangle(x,y)\leq\frac{2\mathrm{Area}(\vv,\underline{w})}{\|\vv\|\|\underline{w}\|},
	$$
	where $\vv, \underline{w}\in\R^2$ are arbitrary non-zero vectors from the subspaces corresponding to $x$ and $y$.
\end{lemma}

The proof of the lemma is straightforward.

\begin{lemma}\label{lstrongholder}
	There exists a constant $C>0$ such that for every $\ii,\jj\in\Sigma$
	$$
	\sphericalangle(e^{ss}(\ii),e^{ss}(\jj))\leq Ce^{-\beta(\ii_+\wedge\jj_+)}\text{\ and\ }	\sphericalangle(e^{s}(\ii),e^{s}(\jj))\leq Ce^{-\beta(\ii_-\wedge\jj_-)}
	$$
	where $\beta$ is the domination exponent in Definition~\ref{ddomsplit}. Thus, the maps $\ii_+\in\Sigma^+\mapsto e^{ss}(\ii_+)$ and $\ii_-\in\Sigma^-\mapsto\log\|A_{i_{-1}}|e^s(\sigma_-\ii_-)\|$ are H\"older continuous.
\end{lemma}

\begin{proof}
	We prove only the inequality for $e^{ss}$, for $e^s$ the argument is similar. Fix $\ii,\jj\in\Sigma$ with $\ii_+\wedge\jj_+=n$. Let $\vv\in e^{ss}(\sigma^n\ii_+)$ and $\wv\in e^{ss}(\sigma^n\jj_+)$ be arbitrary such that $\|\vv\|=\|\wv\|=1$. Then by Lemma~\ref{ldistcomp},
	\begin{multline*}
	\sphericalangle(e^{ss}(\ii),e^{ss}(\jj))\leq 2\frac{\mathrm{Area}(A_{i_0}^{-1}\cdots A_{i_{n-1}}^{-1}\vv,A_{i_0}^{-1}\cdots A_{i_{n-1}}^{-1}\wv)}{\|A_{i_0}^{-1}\cdots A_{i_{n-1}}^{-1}|e^{ss}(\sigma^{n}\ii_+)\|\|A_{i_0}^{-1}\cdots A_{i_{n-1}}^{-1}|e^{ss}(\sigma^{n}\jj_+)\|}\leq\\
	2C^2\frac{|\det(A_{i_0}^{-1}\cdots A_{i_{n-1}}^{-1})|}{\|A_{i_0}^{-1}\cdots A_{i_{n-1}}^{-1}\|^2}\mathrm{Area}(\vv,\wv)\leq 2C^2e^{-\beta n}.
	\end{multline*}
\end{proof}

Let $\varphi:\Sigma^-\mapsto\R$ be a H\"older continuous potential function. Then there exist a constants $C>0, P\in\R$ and $\sigma_-$-invariant Borel probability measures $\mu_-$ and $\mu$ on $\Sigma^-$ and $\Sigma$ such that
\begin{equation}\label{edefGibbs-}
C^{-1}\leq\dfrac{\mu_-([\ii_-|_{-n}^{-1}])}{e^{-nP+\sum_{k=0}^{n-1}\varphi(\sigma_-^k\ii_-)}}\leq C,\text{ for every $\ii_-\in\Sigma^-$,}
\end{equation}
\begin{equation}\label{edefGibbs}
C^{-1}\leq\dfrac{\mu([\ii|_{-n}^{-1}])}{e^{-nP+\sum_{k=0}^{n-1}\varphi(\sigma_-^k\ii)}}\leq C,\text{ for every $\ii\in\Sigma$.}
\end{equation}
We call the measures $\mu_-$ and $\mu$ the \textit{Gibbs measures} of the potential $\varphi$ on $\Sigma^-$ and $\Sigma$. Moreover, $\mu_-$ and $\mu$ are ergodic, see \cite[Chapter~1]{Bbook}. Let $\nu=(\pi^-)_*\mu_-$, where $\pi^-$ is defined in \eqref{enatproj}. Let us denote the projection from $\Sigma$ to $\Sigma^+$ by $p_+:\Sigma\mapsto\Sigma^+$, and similarly, the projection from $\Sigma$ to $\Sigma^-$ by $p_-:\Sigma\mapsto\Sigma^-$. It is easy to see that $(p_-)_*\mu=\mu_-$.

\begin{lemma}\label{lmu+}
	The measure $\mu_+:=(p_+)_*\mu$ is $\sigma$-invariant, ergodic quasi-Bernoulli measure on $\Sigma^+$ with entropy $h_{\mu_+}=h_{\mu}=h_{\mu_-}=P-\int\varphi(\ii)d\mu(\ii)$.
\end{lemma}

We call a measure $m$ on $\Sigma^+$ quasi-Bernoulli, if there exists a uniform contant $C>0$ such that for every $\iv,\jv\in\Sigma^*$
$$
C^{-1}\nu([\iv])\nu([\jv])\leq\nu([\iv\jv])\leq C\nu([\iv])\nu([\jv]),
$$
where $\iv\jv$ is the concatenation of $\iv$ and $\jv$.

\begin{proof}
	First, we prove invariance. Let $A\subseteq\Sigma^+$ be measurable set. Then by using that $\mu$ is $\sigma$-invariant we get
	$$
	\mu_+(\sigma^{-1}A)=\mu_+\left(\bigcup_{i=1}^NiA\right)=\mu\left(\Sigma^-\times\bigcup_{i=1}^NiA\right)=\mu(\Sigma^-\times A)=\mu_+(A).
	$$
	Let $A\subseteq\Sigma^+$ be an arbitrary $\sigma$-invariant subset of $\Sigma^+$. Then $\sigma^{-1}\Sigma^-\times A=\Sigma^{-}\times\left(\bigcup_{i=1}^NiA\right)=\Sigma^-\times\sigma^{-1}A=\Sigma^-\times A$. Therefore, $\mu(\Sigma^-\times A)=0\text{ or }1$, which implies the ergodicity of $\mu_+$.
	
	Finally, let $(i_0,\dots,i_{n+m+1})\in(\Sigma^+)^*$ be arbitrary and let $\jj\in\Sigma^-$ be such that $j_{-1}=i_{n+m+1},\dots,$ $j_{-(n+m+2)}=i_0$. Then by \eqref{edefGibbs}
	\begin{eqnarray*}
	&&\mu_+([i_0,\dots,i_{n+m+1}])=\mu(\Sigma^-\times[i_0,\dots,i_{n+m+1}])=\\
&&\mu([\jj|_{-(n+m+2)}^{-1}])\leq Ce^{\scalebox{1}{$-(n+m+2)P+\sum_{k=0}^{n+m+1}\varphi(\sigma_-^k\jj)$}}= \\
	 &&Ce^{\scalebox{1}{$-(n+1)P+\sum_{k=0}^{n}\varphi(\sigma_-^k\jj)$}}e^{\scalebox{1}{$-(m+1)P+\sum_{k=0}^{m}\varphi(\sigma_-^k(\sigma_-^{n+1}\jj))$}}\leq \\
	&&C^{3}\mu([\jj|_{-(n+1)}^{-1}])\mu([\sigma_-^{n+1}\jj|_{-(m+1)}^{-1}])=C^{3}\mu(\Sigma^-\times[i_0,\dots,i_n])\mu(\Sigma^-\times[i_{n+1},\dots,i_{n+m+1}])=\\
&&C^{3}\mu_+([i_0,\dots,i_n])\mu_+([i_{n+1},\dots,i_{n+m+1}]).
	\end{eqnarray*}

	The inequality $\mu_+([i_0,\dots,i_{n+m+1}])\geq C^{-3}\mu_+([i_0,\dots,i_n])\mu_+([i_{n+1},\dots,i_{n+m+1}])$ can be proven similarly. By using the definition of entropy, see \cite[Theorem~4.10, Theorem~4.18]{W},
	\begin{multline*}
	h_{\mu_+}=\lim_{n\rightarrow\infty}-\frac{1}{n}\sum_{\iv\in\mathcal{S}^n}\mu_+([\iv])\log\mu_+([\iv])\leq P- \lim_{n\rightarrow\infty}\frac{1}{n}\sum_{\iv\in\mathcal{S}^n}\mu_+([\iv])\varphi(\icv\jj)=\\
	P- \lim_{n\rightarrow\infty}\frac{1}{n}\sum_{\iv\in\mathcal{S}^n}\mu_-([\iv])\varphi(\iv\jj)=P-\int\varphi(\ii)d\mu(\ii).
	\end{multline*}
\end{proof}

By Oseledec's multiplicative ergodic theorem, there are constants $0<\chi_{\mu}^s\leq\chi_{\mu}^{ss}$ that
\begin{eqnarray*}
&&\lim_{n\rightarrow\infty}-\frac{1}{n}\log\alpha_1(A_{i_0}\cdots A_{i_{n-1}})=\chi_{\mu}^s\text{ and}\\
&&\lim_{n\rightarrow\infty}-\frac{1}{n}\log\alpha_2(A_{i_0}\cdots A_{i_{n-1}})=\chi_{\mu}^{ss}\text{ for $\mu$-a.e. $\ii\in\Sigma$ ( or $\mu_+$-a.e $\ii_+\in\Sigma^+$).}		
\end{eqnarray*}
We call the values $\chi_{\mu}^s$ the \textit{stable} and $\chi_{\mu}^{ss}$ the \textit{strong stable Lyapunov exponent} of $\mu$. We define the Lyapunov exponents for $\mu_-$ similarly.

Now we define the H\"older continuous potential function and the corresponding Gibbs measure motivated by the singular value function. This measure is our candidate to be the dimension maximizing measure.

\begin{definition}\label{dKaenmaki}
	Let $\Alpha=\left\{A_1, A_2,\dots, A_N\right\}$ be a finite set of contracting, non-singular $2\times2$ matrices such that $\Alpha$ satisfies the dominated splitting. Moreover, let $s_0=s_0(\mathcal{A})$ be the unique root of the subadditive pressure \eqref{esap}. We define $\varphi:\Sigma^-\mapsto\R$ be H\"older continuous potential function as follows,
	\begin{equation}\label{eKaenmakipotential}
		\varphi(\ii_-)=\left\{\begin{array}{cc}
		\log\|A_{i_{-1}}|e^s(\sigma_-\ii_-)\|^{s_0} & \text{ if $0\leq s_0\leq1$,}\\
		\log\left(|\det A_{i_{-1}}|^{s_0-1}\|A_{i_{-1}}|e^s(\sigma_-\ii_-)\|^{2-s_0}\right) & \text{ if $1< s_0<2$.}
		\end{array}\right.
	\end{equation}
	Then we call the Gibbs measure $\mu^K$ with potential $\varphi$ the \textnormal{K\"aenm\"aki measure} on $\Sigma^-$. In particular, there exists a constant $C>0$ such that
	\begin{equation*}
	C^{-1}\leq\dfrac{\mu^K([\ii_-|_{-n}^{-1}])}{\phi^{s_0}(A_{i_{-1}}\cdots A_{i_{-n}})}\leq C,\text{ for every $\ii_-\in\Sigma^-$,}
	\end{equation*}
	where $\phi^s$ is the singular value function \eqref{esvf}.
\end{definition}

Observe that $\mathrm{exp}(\sum_{k=0}^{n-1}\varphi(\sigma^n_-\ii_-))$ is essentially $\phi^{s_0}(A_{i_{-1}}\cdots A_{i_{-n}})$ (defined in \eqref{esvf}), where $s_0$ is the unique root of the subadditive pressure function \eqref{esap}. That is by Lemma~\ref{ldomsplit}, if $s_0\leq1$ then for every $n\geq1$, $\phi^{s_0}(A_{i_{-1}}\cdots A_{i_{-n}})\approx \|A_{i_{-1}}\cdots A_{i_{-n}}|e^s(\sigma_-^n\ii_-)\|^{s_0}=\mathrm{exp}(\sum_{k=0}^{n-1}\varphi(\sigma^n_-\ii_-))$. On the other hand, if $1<s_0<2$ then
\begin{multline*}
\phi^{s_0}(A_{i_{-1}}\cdots A_{i_{-n}})=\alpha_1(A_{i_{-1}}\cdots A_{i_{-n}})\alpha_2((A_{i_{-1}}\cdots A_{i_{-n}}))^{s_0-1}=\\
\left(\alpha_1(A_{i_{-1}}\cdots A_{i_{-n}})\alpha_2((A_{i_{-1}}\cdots A_{i_{-n}}))\right)^{s_0-1}\alpha_1(A_{i_{-1}}\cdots A_{i_{-n}})^{2-s_0}\approx\\
\det(A_{i_{-1}}\cdots A_{i_{-n}})^{s_0-1}\|A_{i_{-1}}\cdots A_{i_{-n}}|e^s(\sigma_-^n\ii_-)\|^{2-s_0}=\mathrm{exp}(\sum_{k=0}^{n-1}\varphi(\sigma^n_-\ii_-)).
\end{multline*}
The H\"older continuity of potential $\varphi$ in \eqref{eKaenmakipotential} follows by Lemma~\ref{lstrongholder}. Basically, the dominated splitting condition (Definition~\ref{ddomsplit}) allows us to show that the potential $\varphi$ is H\"older, hence the measure $\mu^K$ is Gibbs. Without dominated splitting the map $\ii \mapsto \log \|A_{i_{-1}}|e^s(\sigma_-\ii)\|$ is in general only measureable (by Oseledec Theorem).

\section{Ledrappier-Young formula for Gibbs measures}\label{sLY}

In this section, we extend the result \cite[Theorem~2.7]{B} for Gibbs-measures. For every $\theta\in\mathbf{P}^1$ we denote the orthogonal projection in the direction of $\theta$ by $\proj_{\theta}$. Let us define the transversal measure for every $\ii_+\in\Sigma^+$ by $\nu^T_{\ii_+}=\nu\circ(\pr{\ii_+})^{-1}$. That is, $\nu^T_{\ii_+}$ denotes the orthogonal projection of the measure $\nu$ along the line $e^{ss}(\ii_+)$.

\begin{theorem}\label{tLYGibbs}
	Let $\Alpha=\left\{A_1, A_2,\dots, A_N\right\}$ be a finite set of contracting, non-singular $2\times2$ matrices, and let $\Phi=\left\{f_i(\xv)=A_i\xv+\tv_i\right\}_{i=1}^N$ be an iterated function system on the plane with affine mappings. Let $\mu_-$ be a right-shift invariant and ergodic Gibbs measure on $\Sigma^{-}$ defined in \eqref{edefGibbs-}, and $\nu=(\pi^-)_*\mu_-$ be the push-down measure of $\mu_-$. If
	\begin{enumerate}
		\item $\Alpha$ satisfies the dominated splitting,
		\item $\Phi$ satisfies the strong separation condition
	\end{enumerate} then $\nu$ is exact dimensional and
	\begin{equation*}
	\dim_H\nu=\frac{h_{\mu}}{\chi^s_{\mu}}+\left(1-\frac{\chi^{s}_{\mu}}{\chi^{ss}_{\mu}}\right)\dim_H\nu^{T}_{\ii_+}\text{ for $\mu_+$-almost every $\ii_+\in\Sigma^+$.}
	\end{equation*}
\end{theorem}

During the proof of Theorem~\ref{tLYGibbs}, we follow the proof of \cite[Theorem~2.7]{B}. The proof of \cite[Theorem~2.7]{B} is decomposed into four propositions \cite[Proposition~3.1, Proposition~3.3, Proposition~3.8 and Proposition~3.9]{B}. However, \cite[Proposition~3.1]{B} and \cite[Proposition~3.9]{B} hold for general ergodic measures. On the other hand, \cite[Proposition~3.8]{B} follows from \cite[Proposition~3.3]{B} exactly in the same way for Gibbs measures as for Bernoulli measures. So, we extend in the rest of the section \cite[Proposition~3.3]{B} for Gibbs measures.

Let $F$ be the dynamical system defined in \cite[Section~3]{B} acting on $\clo\times\Sigma^{+}$. Namely,
\begin{equation*} 
F(\xv,\ii):=(f_{i_0}(\xv),\sigma\ii),
\end{equation*}
where $\mathcal{O}$ is an open and bounded set such that
$$
\bigcup_{i=1}^Nf_i(\mathcal{O})\subseteq\mathcal{O}\text{ and }f_i(\overline{\mathcal{O}})\cap f_j(\overline{\mathcal{O}})=\emptyset\text{ for $i\neq j$.}
$$
Since $F$ is a hyperbolic map acting $\clo\times\Sigma^{+}$, its unique non-empty and compact $F$-invariant set is $\bigcap_{n=0}^{\infty}F^n(\clo\times\Sigma^+)=\Lambda\times\Sigma^+$. It is easy to see that $F$ is conjugate to $\sigma$ by the projection $\pi:\Sigma\mapsto\Lambda\times\Sigma^{+}$, where $\pi(\ii):=(\pi^-(\ii_-),\ii_+)$. That is,
$\pi\circ\sigma=F\circ\pi.$ Denote the measure $\pi_*\mu$ by $\widehat{\nu}$. Then $\widehat{\nu}$ is $F$-invariant ergodic measure.

Since $e^{ss}$ depends only on $\ii_+$, it defines a foliation on $\clo$ for every $\ii_+\in\Sigma^+$. Hence, it defines a foliation $\xi^{ss}$ on $\Lambda\times\Sigma^+$. Namely, for a $\y=(\xv,\ii_+)\in\Lambda\times\Sigma^+$ let $l_{ss}(\y)$ be the line through $\xv$ parallel to $e_{ss}(\ii_+)$ on $\R^2\times\left\{\ii_+\right\}$. Let the partition element $\xi^{ss}(\y)$ be the intersection of the line $l_{ss}(\y)$ with $\Lambda\times\left\{\ii_+\right\}$. Denote by $F\xi^{ss}$ the image of the partition $\xi^{ss}$ under $F$, i.e. for every $\y$, $(F\xi^{ss})(\y)=F(\xi^{ss}(F^{-1}(\y)))$. It is easy to see that $F\xi^{ss}$ is a refinement of $\xi^{ss}$, that is, for every $\y$, $(F\xi^{ss})(\y)\subset\xi^{ss}(\y)$.

We decompose the measure $\widehat{\nu}$ on $\Lambda\times\Sigma^+$according to two different partitions. First, we construct a family of measures supported on $\Lambda$. More precisely, supported on $\Lambda\times\{\ii_+\}$ for $\mu_+$-a.e. $\ii_+$. So, applying Rokhlin's Theorem \cite{R}, for $\mu_+$-a.e. $\ii_+\in\Sigma^+$ there exists a uniquely defined system of conditional measures $\mu_{\ii_+}$ up to a set of zero measure, supported on $\Sigma^-\times\left\{\ii_+\right\}$ and
\begin{equation*}
\mu(A)=\int\mu_{\ii_+}(A)d\mu_+(\ii_+).
\end{equation*}
By defining $\widehat{\nu}_{\ii_+}:=(\pi^-)_*\mu_{\ii_+}$, we get
$$
\widehat{\nu}=\int\widehat{\nu}_{\ii_+}d\mu_+(\ii_+).
$$
In the focus of our study stand the geometric measure theoretical properties of the family of measures $\widehat{\nu}_{\ii_+}$ along the strong stable directions. Therefore, first we define the transversal measure, i.e. for $\mu_+$-a.e. $\ii_+\in\Sigma^+$, let $\widehat{\nu}^T_{\ii_+}$ be the orthogonal projection of $\widehat{\nu}_{\ii_+}$ along the subspace $e^{ss}(\ii_+)$. That is, $$\widehat{\nu}^T_{\ii_+}:=(\pr{\ii_+})_*\widehat{\nu}_{\ii_+}.$$

On the other hand, we need the conditional measures of $\widehat{\nu}_{\ii_+}$ along the subspace $e^{ss}(\ii_+)$. Applying Rokhlin's Theorem \cite{R} again, there exists a canonical system of conditional measures, i.e. for $\widehat{\nu}$-a.e. $\y\in\Lambda\times\Sigma^+$ there exists a measure $\widehat{\nu}^{ss}_{\y}$ supported on $\xi^{ss}(\y)$ such that the measures are uniquely defined up to a zero measure set of $\y$ and for every measurable set $A$ the function $\y\mapsto\widehat{\nu}_{\y}^{ss}(A)$ is measurable. Moreover,
\begin{equation}\label{econdmeasure}
\widehat{\nu}(A)=\int\widehat{\nu}_{\y}^{ss}(A)d\widehat{\nu}(\y).
\end{equation}
By the uniqueness of the conditional measures, we get that the measure $\widehat{\nu}^{ss}_{\y}$ is conditional measure of $\widehat{\nu}_{\ii_+}$, namely,
$$
\widehat{\nu}_{\ii_+}=\int\widehat{\nu}_{(\xv,\ii_+)}^{ss}d\widehat{\nu}^T_{\ii_+}(\xv)\text{ for $\mu_+$-a.e. $\ii_+\in\Sigma^+$.}
$$

Let us define the conditional entropy of $F\xi^{ss}$ with respect to $\xi^{ss}$ in the usual way,
\begin{equation*} 
H(F\xi^{ss}|\xi^{ss}):=-\int\log\widehat{\nu}_{\y}^{ss}((F\xi^{ss})(\y))d\widehat{\nu}(\y).
\end{equation*}

One of the main goals of this paper is to show that there is a dimension maximizing Gibbs measure for self-affine sets. However, our method allows us only to handle the dimension of the conditional measures $\mu_{\ii_+}$. The next lemma is devoted to show that $\mu_{\ii_+}$ is not necessarily equal to but equivalent with a Gibbs measure on $\Sigma^-$.

\begin{lemma}\label{lequiv}
	There exists a constant $C>0$ such that $C^{-1}\mu_-\times\mu_+\leq\mu\leq C\mu_-\times\mu_+$. In particular,
	\begin{equation}\label{eeqvcond}
	C^{-1}\mu_-\leq\mu_{\ii_+}\leq C\mu_-\text{ for $\mu_+$-a.e. $\ii_+\in\Sigma^+$.}
	\end{equation}
\end{lemma}

\begin{proof}
	It is enough to show that there exists a $C>0$ such that for every $\ii\in\Sigma$ and $n,m\geq0$
	$$
	C^{-1}\mu_-([\ii|_{-n}^{-1}])\mu_+([\ii|_0^m])\leq\mu([\ii|_{-n}^m])\leq C\mu_-([\ii|_{-n}^{-1}])\mu_+([\ii|_0^m]).
	$$
	Indeed, every set $A$ in the $\sigma$-algebra can be approximated by cylinder sets. By the definition of Gibbs measure $\mu$
	\begin{multline*}
	\mu([\ii|_{-n}^m])=\mu([\sigma^{m+1}\ii|_{-(n+m+1)}^{-1}])\leq C e^{-(n+m+1)P+\sum_{k=0}^{n+m}\varphi(\sigma_-^k\sigma^{m+1}\ii)}=\\
	C e^{-nP+\sum_{k=0}^{n-1}\varphi{\sigma_-^k\ii}}e^{-(m+1)P+\sum_{k=0}^m\varphi(\sigma_-^k\sigma^{m+1}\ii)}\leq C^2\mu_-([\ii|_{-n}^{-1}])\mu([\sigma^{m+1}\ii|_{-(m+1)}^{-1}])=\\
	C^2\mu_-([\ii|_{-n}^{-1}])\mu([\ii|_{0}^{m}])=C^2\mu_-([\ii|_{-n}^{-1}])\mu_+([\ii|_{0}^{m}]).
	\end{multline*}
	The other inequality can be proven similarly. The relation \eqref{eeqvcond} follows by the fact that the conditional measures are uniquely defined up to a set of zero measure.
\end{proof}

By Lemma~\ref{lequiv}, the measures $\widehat{\nu}_{\ii_+}$ and $\nu$ are equivalent for $\mu_+$-a.e. $\ii_+\in\Sigma^+$. Similarly,  the measures $\widehat{\nu}_{\ii_+}^T$ and $\nu_{\ii_+}^T$ are equivalent for $\mu_+$-a.e. $\ii_+\in\Sigma^+$.

For the examination of the local dimension of the projected measure, instead of looking at balls on lines we introduce the transversal stable balls associated to the projection. Let $B^t_r(\xv,\ii)$ be transversal stable ball with radius $r$, i.e $$B^t_r(\xv,\ii)=\left\{(\yv,\jj):\ii=\jj\ \&\ \mathrm{dist}(l_{ss}(\xv,\ii),l_{ss}(\yv,\jj))\leq2r\right\},$$
where $l_{ss}(\xv,\ii)$ denotes the line through $\xv$ parallel to $e_{ss}(\ii)$. Here, $\mathrm{dist}(.,.)$ is the usual Euclidean distance between parallel lines.

For technical reasons, we also have to introduce the modified transversal stable ball. Since the IFS $\Phi$ satisfies the SSC, for an $\y=(\xv,\ii)\in\Lambda\times\Sigma^+$ we can define the stable direction $e_s(\y)$ of $\y$ by $e_s(\y):=e_s(\xv):=e_s(\ii_-)$, where $\pi_-(\ii_-)=\xv$. Denote $\mathrm{dist}_{e_s(\y)}$ the natural Euclidean distance on the subspace $e_s(\y)$.

Then for an $(\xv,\ii)\in\Lambda\times\Sigma^+$, we define the modified transversal stable ball with radius $\delta$ by
\begin{equation*}
B^T_{\delta}(\xv,\ii)=\left\{(\yv,\jj)\in\Lambda\times\Sigma^+:\ii=\jj\ \&\ \mathrm{dist}_{e_s(\xv,\ii)}(l_{ss}(\xv,\ii),l_{ss}(\yv,\jj))\leq\delta\right\},
\end{equation*}
where $\mathrm{dist}_{e_s(\xv,\ii)}(l_{ss}(\xv,\ii),l_{ss}(\yv,\jj))$ means the distance of the intersections of the lines $l_{ss}(\xv,\ii),l_{ss}(\yv,\jj)$ with the subspace $e_s(\xv,\ii)$ with respect to the distance $\mathrm{dist}_{e_s(\xv,\ii)}$. Since there exists a constant $\alpha>0$ such that
\begin{equation*}
\sphericalangle(e_s(\ii_-),e_{ss}(\ii_+))\geq\alpha>0,\text{ for every $\ii_-\in\Sigma^-$ and $\ii_+\in\Sigma^+$,}
\end{equation*}
there exists a constant $c>0$ that for every $\y\in\Lambda\times\Sigma^+$ and $r>0$
\begin{equation}\label{ecomptransballs}
B^T_{c^{-1}r}(\xv,\ii)\subseteq B^t_{r}(\xv,\ii)\subseteq B^T_{cr}(\xv,\ii).
\end{equation}

We are going to prove the following proposition.

\begin{prop}\label{pLY2}
	For $\mu_+$-a.e. $\ii_+\in\Sigma^+$ the measure $\nu^T_{\ii_+}$ is exact dimensional and $$\dim_H\nu^T_{\ii_+}=\frac{h_{\mu}-H(F\xi^{ss}|\xi^{ss})}{\chi^s_{\mu}}.$$
	In particular,
	$$
	\lim_{r\to0+}\frac{\nu(B^T_{r}(\xv,\ii_+))}{\log r}=\frac{h_{\mu}-H(F\xi^{ss}|\xi^{ss})}{\chi^s_{\mu}}\text{ for $\widehat{\nu}$-a.e. $(\xv,\ii_+)$}.
	$$
\end{prop}

Let $\mathcal{P}$ be the natural partition, i.e. $\mathcal{P}=\left\{f_i(\Lambda)\times\Sigma^{+}\right\}_{i=1}^N$. Denote the $k$th refinement of $\mathcal{P}$ by $\mathcal{P}_1^k$, i.e. for every $\y\in\Lambda\times\Sigma^+$,  $\mathcal{P}_1^k(\y)=\left(\bigvee_{i=1}^kF^i(\mathcal{P})\right)(\y)=\mathcal{P}(\y)\cap F(\mathcal{P}(F^{-1}(\y)))\cap\dots\cap F^k(\mathcal{P}(F^{-k}(\y)))$. In other words, $\mathcal{P}_1^k$ is the standard partition into $k$-level cylinders.

Let us define almost everywhere the measurable functions $g_k(\y):=\widehat{\nu}_{\y}^{ss}(\mathcal{P}_1^k(\y))$ and
$$
g_{\delta,k}(\y):=\frac{\widehat{\nu}_{\ii_+}(B^T_{\delta}(\y)\cap\mathcal{P}_1^k(\y))}{\widehat{\nu}_{\ii_+}(B^T_{\delta}(\y))}.
$$
By definition, $g_{\delta,k}(\y)$ is the $\delta$ approximation of the measure of $\mathcal{P}_1^k(\y)$ according to the conditional measure. By Rokhlin's Theorem, $g_{\delta,k}\rightarrow g_k$ as $\delta\rightarrow0+$ for $\widehat{\nu}$ almost everywhere and, since $0\leq g_{\delta,k}\leq1$, \eqref{econdmeasure} implies $g_{\delta,k}\rightarrow g_k$ in $L^1(\widehat{\nu})$ as $\delta\rightarrow0+$.

\begin{lemma}\label{ltoLY2b}
	The function $\sup_{\delta>0}\left\{-\log g_{\delta,k}\right\}$ is in $L^1(\widehat{\nu})$ for every $k\geq1$.
\end{lemma}

The proof of Lemma~\ref{ltoLY2b} coincides with \cite[Lemma~3.6]{B}.

\begin{lemma}\label{linv}
	For every $\xv=\pi^-(i_{-1},i_{-2},\dots)\in\Lambda$, $\ii_+\in\Sigma^+$, $\delta>0$ and $k\geq1$
	$$F^k\left(B^T_{\delta}(F^{-k}(\y))\times[i_{-k},\dots,i_{-1}]\right)=\left(B^T_{\|A_{i_{-1}}\cdots A_{i_{-k}}|e_s(F^{-k}(\y))\|\delta}(\y)\cap\mathcal{P}_1^k(\y)\right)\times\Sigma^+,$$
	where $\y=(\xv,\ii_+)$.
\end{lemma}

By using the fact that $\nu=(\pi^-)_*\mu_-=(\pi^-)_*(p_-)_*\mu$, we have
\begin{multline*}
\nu(B^T_{\delta}(\y)\cap\mathcal{P}_1^k)=\widehat{\nu}\left(B^T_{\delta}(\y)\cap\mathcal{P}_1^k\times\Sigma^+\right)=\\
\widehat{\nu}\left(F^{-k}\left(B^T_{\delta}(\y)\cap\mathcal{P}_1^k\times\Sigma^+\right)\right)=\widehat{\nu}\left(B^T_{\|A_{i_{-1}}\cdots A_{i_{-k}}|e_s(F^{-k}(\y))\|^{-1}\delta}(F^{-k}(\y))\times[i_{-k},\dots,i_{-1}]\right),
\end{multline*}
where in the last equation we used Lemma~\ref{linv}. By Lemma~\ref{lequiv},
\begin{multline}\label{eprodform}
\nu(B^T_{\delta}(\y)\cap\mathcal{P}_1^k(\y))=\widehat{\nu}\left(B^T_{\|A_{i_{-1}}\cdots A_{i_{-k}}|e_s(F^{-k}(\y))\|^{-1}\delta}(F^{-k}(\y))\times[i_{-k},\dots,i_{-1}]\right)\leq\\
C \nu\left(B^T_{\|A_{i_{-1}}\cdots A_{i_{-k}}|e_s(F^{-k}(\y))\|^{-1}\delta}(F^{-k}(\y))\right)\mu_+([i_{-k},\dots,i_{-1}]),
\end{multline}
and
\begin{equation}\label{eprodform2}
\nu(B^T_{\delta}(\y)\cap\mathcal{P}_1^k(\y))\geq C^{-1} \nu\left(B^T_{\|A_{i_{-1}}\cdots A_{i_{-k}}|e_s(F^{-k}(\y))\|^{-1}\delta}(F^{-k}(\y))\right)\mu_+([i_{-k},\dots,i_{-1}])
\end{equation}
for every $\delta>0$, $k\geq1$, and $\y\in\Lambda\times\Sigma^+$.

\begin{proof}[Proof of Proposition~\ref{pLY2}]
	By the definition of the transversal measure, the statement of the proposition is equivalent to
	\begin{equation*}
	\lim_{\delta\rightarrow0+}\frac{\log\nu(B_{\delta}^t(\xv,\ii_+))}{\log\delta}=\frac{h_{\nu}-H(F\xi^{ss}|\xi^{ss})}{\chi^{s}_{\mu}}\text{ for $\nu\times\mu_+$-a.e $(\xv,\ii_+)$.}
	\end{equation*}

	Hence, by \eqref{ecomptransballs} and by Lemma~\ref{ldomsplit}, it is enough to show that if $\y=(\xv,\ii_+)\in\Lambda\times\Sigma^+$ with $\xv=\pi_-(i_{-1},i_{-2},\dots)$,
	\begin{equation*}
	\lim_{p\rightarrow\infty}\frac{\log\nu\left(B^T_{\|A_{i_{-1}}\cdots A_{i_{-pk}}|e_s(F^{-pk}(\y))\|}(\y)\right)}{\log\alpha_1(A_{i_{-1}}\cdots A_{i_{-pk}})}=\frac{h_{\nu}-H(F\xi^{ss}|\xi^{ss})}{\chi^{s}_{\mu}}\text{ for $\nu\times\mu_+$-a.e $\y$.}
	\end{equation*}
	By Oseledec's Theorem, we have
	\begin{equation}\label{emodlyap}
	\lim_{p\to\infty}\frac{1}{p}\log\alpha_1(A_{i_{-1}}\cdots A_{i_{-pk}})=-k\chi^{s}_{\mu}\text{ for $\mu_-$-a.e $\ii_-$.}
	\end{equation}
	
	By applying \eqref{eprodform}, \eqref{eprodform2} and Lemma~\ref{lequiv},
	\begin{eqnarray*}
	&&\nu\left(B^T_{\|A_{i_{-1}}\cdots A_{i_{-pk}}|e_s(F^{-pk}(\y))\|}(\y)\right)=\\
&&\nu\left(B_1^T(F^{-pk})\right)\prod_{l=1}^{p}\dfrac{\nu\left(B^T_{\|A_{i_{-(l-1)k-1}}\cdots A_{i_{-pk}}|e_s(F^{-pk}(\y))\|}(F^{-(l-1)k}(\y))\right)}{\nu\left(B^T_{\|A_{i_{-lk-1}}\cdots A_{i_{-pk}}|e_s(F^{-pk}(\y))\|}(F^{-lk}(\y))\right)}\leq\\
&&	C^p\nu\left(B_1^T(F^{-pk})\right)\prod_{l=1}^{p}\dfrac{\nu\left(B^T_{\|A_{i_{-(l-1)k-1}}\cdots A_{i_{-pk}}|e_s(F^{-pk}(\y))\|}(F^{-(l-1)k}(\y))\right)\mu_+([i_{-(l-1)k-1},\dots,i_{-lk}])}{\nu\left(B^T_{\|A_{i_{-(l-1)k-1}}\cdots A_{i_{-pk}}|e_s(F^{-pk}(\y))\|}(F^{-(l-1)k}(\y))\cap\mathcal{P}_1^k(F^{-(l-1)k}(\y)\right)}\leq\\
&&C^{3p}\nu\left(B_1^T(F^{-pk})\right)\prod_{l=1}^{p}\dfrac{\widehat{\nu}_{F^{-(l-1)k}(\y)}\left(B^T_{\|A_{i_{-(l-1)k-1}}\cdots A_{i_{-pk}}|e_s(F^{-pk}(\y))\|}(F^{-(l-1)k}(\y))\right)\mu_+([i_{-(l-1)k-1},\dots,i_{-lk}])}{\widehat{\nu}_{F^{-(l-1)k}(\y)}\left(B^T_{\|A_{i_{-(l-1)k-1}}\cdots A_{i_{-pk}}|e_s(F^{-pk}(\y))\|}(F^{-(l-1)k}(\y))\cap\mathcal{P}_1^k(F^{-(l-1)k}(\y)\right)}.
\end{eqnarray*}

	Similarly,
	\begin{eqnarray*}
	&&\nu\left(B^T_{\|A_{i_{-1}}\cdots A_{i_{-pk}}|e_s(F^{-pk}(\y))\|}(\y)\right)\geq\\
	&&C^{-3p}\nu\left(B_1^T(F^{-pk})\right)\cdot\\
	&&\prod_{l=1}^{p}\dfrac{\widehat{\nu}_{F^{-(l-1)k}(\y)}\left(B^T_{\|A_{i_{-(l-1)k-1}}\cdots A_{i_{-pk}}|e_s(F^{-pk}(\y))\|}(F^{-(l-1)k}(\y))\right)\mu_+([i_{-(l-1)k-1},\dots,i_{-lk}])}{\widehat{\nu}_{F^{-(l-1)k}(\y)}\left(B^T_{\|A_{i_{-(l-1)k-1}}\cdots A_{i_{-pk}}|e_s(F^{-pk}(\y))\|}(F^{-(l-1)k}(\y))\cap\mathcal{P}_1^k(F^{-(l-1)k}(\y)\right)}.
	\end{eqnarray*}
	By taking logarithm and dividing by $p$ we get
	\begin{multline*}
	\frac{1}{p}\log\nu\left(B_1^T(F^{-pk})\right)-3\log C-\frac{1}{p}\sum_{l=1}^p\log g_{\|A_{i_{-lk-1}}\cdots A_{i_{-pk}}|e_s(F^{-pk}(\y))\|,k}(F^{-lk}(\y))+\\
	\frac{1}{p}\sum_{l=1}^p\log\mu_+([i_{-(l-1)k-1},\dots,i_{-lk}])\leq\frac{1}{p}\log\nu\left(B^T_{\|A_{i_{-1}}\cdots A_{i_{-pk}}|e_s(F^{-pk}(\y))\|}(\y)\right)
\end{multline*}
and
\begin{multline*}
\frac{1}{p}\log\nu\left(B^T_{\|A_{i_{-1}}\cdots A_{i_{-pk}}|e_s(F^{-pk}(\y))\|}(\y)\right)\leq
	\frac{1}{p}\log\nu\left(B_1^T(F^{-pk})\right)+3\log C-\\
\frac{1}{p}\sum_{l=1}^p\log g_{\|A_{i_{-lk-1}}\cdots A_{i_{-pk}}|e_s(F^{-pk}(\y))\|,k}(F^{-lk}(\y))+
	\frac{1}{p}\sum_{l=1}^p\log\mu_+([i_{-(l-1)k-1},\dots,i_{-lk}]).
	\end{multline*}
	By Lemma~\ref{ltoLY2b}, we may apply the result of Maker's Ergodic Theorem~\cite[Theorem~1]{M}, so we get
	$$
	\lim_{p\to\infty}-\frac{1}{p}\sum_{l=1}^p\log g_{\|A_{i_{-lk-1}}\cdots A_{i_{-pk}}|e_s(F^{-pk}(\y))\|,k}(F^{-lk}(\y))=-\int\log g_{k}(\y)d\widehat{\nu}(\y)=kH(F\xi^{ss}|\xi^{ss})
	$$
	for $\widehat{\nu}$-a.e. $\y$. Applying Birkhoff's ergodic theorem and \eqref{emodlyap} we get
	\begin{multline*}
	\frac{-3\log C-kH(F\xi^{ss}|\xi^{ss})-\sum_{\iv\in\mathcal{S}^k}\mu_+([\iv])\log\mu_+([\iv])}{k\chi^s_{\mu}}\leq\underline{d}_{\nu^T_{\ii_+}}(\xv)\leq\overline{d}_{\nu^T_{\ii_+}}(\xv)\leq\\
	\frac{3\log C-kH(F\xi^{ss}|\xi^{ss})-\sum_{\iv\in\mathcal{S}^k}\mu_+([\iv])\log\mu_+([\iv])}{k\chi^s_{\mu}}\text{ for $\widehat{\nu}$-a.e. $\y$ and every $k\geq1$.}
	\end{multline*}
	By taking the limit $k\to\infty$, we get that
	$$
	\underline{d}_{\nu^T_{\ii_+}}(\xv)=\overline{d}_{\nu^T_{\ii_+}}(\xv)=\frac{h_{\mu}-H(F\xi^{ss}|\xi^{ss})}{\chi^s_{\mu}}\text{ for $\widehat{\nu}$-a.e. $\y$}.
	$$
	Since $\widehat{\nu}$ is equivalent to $\nu\times\mu_+$, the statement follows.
\end{proof}

\begin{proof}[Proof of Theorem~\ref{tLYGibbs}]
	Since the proofs of \cite[[Proposition~3.1, Proposition~3.8 and Proposition~3.9]{B} do not use that the examined measure is Bernoulli, one can modify them to show that for $\widehat{\nu}$-a.e. $\y\in\Lambda\times\Sigma^+$ the measure $\widehat{\nu}_{\y}^{ss}$ is exact dimensional and $$\dim_H\widehat{\nu}_{\y}^{ss}=\frac{H(F\xi^{ss}|\xi^{ss})}{\chi_{\mu}^{ss}}.$$
	Moreover,
	$$
	\liminf_{r\to\infty}\frac{\widehat{\nu}_{\ii_+}(B_r(\xv))}{\log r}\geq\frac{H(F\xi^{ss}|\xi^{ss})}{\chi_{\mu}^{ss}}+\frac{h_{\mu}-H(F\xi^{ss}|\xi^{ss})}{\chi^s_{\mu}}\text{ for $\widehat{\nu}$-a.e. $(\xv,\ii_+)$}
	$$
	and by using that $\nu=(p_-)_*\widehat{\nu}$
	$$
	\limsup_{r\to\infty}\frac{\nu(B_r(\xv))}{\log r}\leq\frac{H(F\xi^{ss}|\xi^{ss})}{\chi_{\mu}^{ss}}+\frac{h_{\mu}-H(F\xi^{ss}|\xi^{ss})}{\chi^s_{\mu}}\text{ for $\nu$-a.e. $\xv$}.
	$$
	Since the measure $\nu$ is equivalent to $\widehat{\nu}_{\ii_+}$ for $\mu_+$-a.e. $\ii_+$, the statement follows by Proposition~\ref{pLY2}.
\end{proof}

As a corollary of Theorem~\ref{tLYGibbs}, we are able to give two conditions which ensure that the dimension of a Gibbs measure is equal to its Lyapunov dimension. The second part of condition \eqref{ctdim3} in the next theorem appears in \cite{FK}, as well, for the Gibbs measure generated by the subadditive pressure.

\begin{theorem}\label{tdimGibbs}
	Let $\mathcal{A}=\left\{A_k\right\}_{k=1}^N$ be a family of $2\times2$ real non-singular matrices and $\Phi=\left\{A_k\xv+\tv_k\right\}_{k=1}^N$ be an IFS of affinities on the plane. Moreover, let $\mu_-$ be a $\sigma_-$-invariant ergodic Gibbs measures on $\Sigma^-$, let $\mu$ be its unique extension to $\Sigma$ and let $\mu_+$ be the quasi-Bernoulli measure defined in Lemma~\ref{lmu+}. Assume that
	\begin{enumerate}[(i)]
		\item the IFS $\Phi$ satisfies the strong separation condition,
		\item $\mathcal{A}$ satisfies dominated splitting condition
        \item\label{ctdim3} either $\dim_H(e^{ss})_*\mu_+\geq\min\left\{1,\ldim\mu_-\right\}$ or $\dim_H(e^{ss})_*\mu_++\dim_H(\pi^-)_*\mu_->2$
	\end{enumerate}
	Then
	$$\dim_H(\pi^-)_*\mu=\min\left\{\frac{h_{\mu}}{\chi^{s}_{\mu}},1+\frac{h_{\mu}-\chi^{s}_{\mu}}{\chi^{ss}_{\mu}}\right\}.$$
\end{theorem}

By Theorem~\ref{tLYGibbs}, the proof is similar to the proofs of \cite[Theorem~2.8 and Theorem~2.9]{B}.

\section{Dimension of Gibbs measures and transversality condition of strong stable directions}\label{ssstc}

In this section and the rest of the paper, we are going to study the dimension of Gibbs measures. To be able to calculate the dimension of Gibbs measure, we have to handle the dimension of strong stable directions, see \eqref{ctdim3} of Theorem~\ref{tdimGibbs}. In the case, when the matrices satisfies the backward non-overlapping condition, i.e. there exists a backward invariant multicone $M$ such that $A^{-1}_i(M^{o})\subseteq M^{o}$ and $A^{-1}_i(M^{o})\cap A_j^{-1}(M^{o})=\emptyset$ for every $i\neq j$, it is possible to calculate the dimension of strong stable directions. Namely, by \cite[Lemma~4.2]{B}, for every $\sigma$-invariant ergodic measure $\mu$ on $\Sigma^+$
\begin{equation*}
\dim_H(e^{ss})_*\mu=\frac{h_{\mu}}{\chi^{ss}_{\mu}-\chi^{s}_{\mu}},
\end{equation*}
where $h_{\mu}$ denotes the entropy of $\mu$.

In general a set of matrices does not satisfy this phenomena. In this section we introduce a condition, which makes us able to handle the problem of overlaps.  Namely, we consider a parametrized family of matrices $\mathcal{A}(\lv)$ with the corresponding map of stable- and strong stable directions $e^{s}_{\lv}$ and $e^{ss}_{\lv}$.

\begin{definition}\label{dsstrans}
	Let $U\subset\R^d$ be open and bounded. We say that a parametrized family of matrices $\mathcal{A}(\lv)=\left\{A_i(\lv)\right\}_{i=1}^N$ satisfies the \underline{strong-stable transversality} on $U$ if
	\begin{itemize}
		\item the parametrisation $\lv\mapsto A_i(\lv)$ is continuous for every $i=1,\dots,N$ on an open neighbourhood of $\overline{U}$
		\item for every $\lv\in\overline{U}$ the set $\mathcal{A}(\lv)$ satisfies the dominated splitting
		\item there exists a constant $C>0$ that for every $\ii,\jj\in\Sigma^+$ with $i_0\neq j_0$
		$$
		\mathcal{L}_d\left\{\lv\in U:\sphericalangle(e^{ss}_{\lv}(\ii),e^{ss}_{\lv}(\jj))<r\right\}\leq Cr\text{ for every $r>0$}.
		$$
	\end{itemize}
\end{definition}

The definition of strong-stable transversality is a natural generalisation of the transversality condition for iterated function systems, see \cite[(2.9)]{SSU}.

\begin{theorem}\label{tdimGibbs2}
	Let $U\subset\R^d$ be an open and bounded set and let $\mathcal{A}(\lv)=\left\{A_k(\lv)\right\}_{k=1}^N$ be a parametrized family of $2\times2$ real matrices and $\Phi(\lv)=\left\{A_k(\lv)\xv+\tv_k(\lv)\right\}_{k=1}^N$ be a parametrized family affine IFSs on the real plane such that
	\begin{enumerate}[(i)]
		\item for every $\lv\in U$ the IFS $\Phi(\lv)$ satisfies the strong separation condition,
		\item $\mathcal{A}(\lv)$ satisfies the strong-stable transversality on $U$.
	\end{enumerate}
	Let $\left\{\mu_{\lv}\right\}_{\lv\in U}$ be a parametrized family of $\sigma_-$-invariant ergodic Gibbs measures on $\Sigma^-$ such that the family of the corresponding H\"older continuous potential functions $\left\{\phi_{\lv}\right\}_{\lv\in U}$ is uniformly continuously parametrized, moreover,
	\begin{enumerate}[(i)]\setcounter{enumi}{2}
	\item\label{ctdim32} $
	\text{either } \dfrac{h_{\mu_{\lv}}}{\chi^{ss}_{\mu_{\lv}}(\lv)-\chi^{s}_{\mu_{\lv}}(\lv)}\geq\min\left\{1,\dfrac{h_{\mu_{\lv}}}{\chi^{s}_{\mu_{\lv}}(\lv)}\right\}\text{\ or\ }\dfrac{h_{\mu_{\lv}}}{\chi^{ss}_{\mu_{\lv}}(\lv)-\chi^{s}_{\mu_{\lv}}(\lv)}+2\dfrac{h_{\mu_{\lv}}}{\chi^{ss}_{\mu_{\lv}}(\lv)}>2
	$
	\end{enumerate}
	Then
	$$\dim_H(\pi^-_{\lv})_*\mu_{\lv}=\min\left\{\frac{h_{\mu_{\lv}}}{\chi^{s}_{\mu_{\lv}}(\lv)},1+\frac{h_{\mu_{\lv}}-\chi^{s}_{\mu_{\lv}}(\lv)}{\chi^{ss}_{\mu_{\lv}}(\lv)}\right\}\text{ for $\mathcal{L}_d$-a.e. $\lv\in U$.}$$
\end{theorem}

The proof of Theorem~\ref{tdimGibbs2} is based on the combination of Theorem~\ref{tdimGibbs} and the following theorem.

\begin{theorem}\label{tss}
	Let $U\subset\R^d$ be an open and bounded set and let $\mathcal{A}(\lv)=\left\{A_k(\lv)\right\}_{k=1}^N$ be a parametrized family of $2\times2$ real matrices such that $\mathcal{A}(\lv)$ satisfies the strong-stable transversality on $U$. Moreover, let $\left\{\mu_{\lv}\right\}_{\lv\in U}$ be a family of $\sigma$-invariant quasi-Bernoulli ergodic measures on $\Sigma^+$ such that $\lv\mapsto h_{\mu_{\lv}}$ is continuous and for every $\lv_0\in U$ and $\varepsilon>0$ there exists a $\delta=\delta(\varepsilon,\lv_0)>0$ that for every $\ii\in\Sigma$, every $n\geq1$ and every $\|\lv-\lv_0\|<\delta$
	\begin{equation}\label{econtmeasure}
	e^{-\varepsilon n}\leq\frac{\mu_{\lv}([\ii|_0^{n-1}])}{\mu_{\lv_0}([\ii|_0^{n-1}])}\leq e^{\varepsilon n}.
	\end{equation}
	Then
	\begin{equation*}
	\dim_H(\e{\lv}{ss})_*\mu_{\lv}=\min\left\{\frac{h_{\mu_{\lv}}}{\chi^{ss}_{\mu_{\lv}}(\lv)-\chi^{s}_{\mu_{\lv}}(\lv)},1\right\}\text{ for $\mathcal{L}_d$-a.e $\lv\in U$.}
	\end{equation*}
\end{theorem}

The proof uses the standard transversality method but for completeness we present it here. First, we give an upper bound for the dimension.

\begin{lemma}\label{lubss}
	Let $\mathcal{A}=\left\{A_i\right\}_{i=1}^N$ be a set of matrices satisfying the dominated splitting and let $e^{ss}:\Sigma^+\mapsto\mathbb{P}^1$ be the map to strong-stable directions. Then for every $\sigma$-invariant ergodic measure $\mu$ on $\Sigma^+$,
	$$
	\dim_H(e^{ss})_*\mu\leq\min\left\{1,\frac{h_{\mu}}{\chi^{ss}_{\mu}-\chi^{s}_{\mu}}\right\}.
	$$
\end{lemma}

\begin{proof}[Proof of Lemma~\ref{lubss}]
	For any $x\in\mathbf{P}^1$ let $B_r^{\sphericalangle}(x):=\left\{y\in\mathbf{P}^1:\sphericalangle(x,y)<r\right\}$. It is enough to show that
	\begin{equation*}
	\liminf_{r\rightarrow0+}\frac{\log(e^{ss})_*\mu(B_r^{\sphericalangle}(e^{ss}(\ii)))}{\log r}\leq\frac{h_{\mu}}{\chi^{ss}_{\mu}-\chi^{s}_{\mu}}\text{ for $\mu$-a.e. }\ii\in\Sigma^+.
	\end{equation*}
	By Lemma~\ref{ldistcomp} and Lemma~\ref{ldomsplit}\eqref{ldomsplit2}, if $\ii, \jj\in\Sigma^+$ that $i_k=j_k$ for $k=0,\dots,n$
	$$
	\sphericalangle(e^{ss}(\ii),e^{ss}(\jj))\leq\dfrac{\mathrm{Area}(A_{i_0}^{-1}\cdots A_{i_n}^{-1}\underline{v},A_{i_0}^{-1}\cdots A_{i_n}^{-1}\underline{w})}{\|A_{i_0}^{-1}\cdots A_{i_n}^{-1}|e^{ss}(\sigma^{n+1}\jj)\|\|A_{i_0}^{-1}\cdots A_{i_n}^{-1}|e^{ss}(\sigma^{n+1}\ii)\|}\leq C\dfrac{|\det(A_{i_0}^{-1}\cdots A_{i_n}^{-1})|}{\|A_{i_0}^{-1}\cdots A_{i_n}^{-1}\|^2},
	$$
	where $\underline{v}\in e^{ss}(\sigma^{n+1}\ii)$ and $\underline{w}\in e^{ss}(\sigma^{n+1}\jj)$ such that $\|\underline{v}\|=\|\underline{w}\|=1$.
	Let $n(r,\ii)\in\N$ be the smallest number such that
	$$
	\dfrac{|\det(A_{i_0}^{-1}\cdots A_{i_n}^{-1})|}{\|A_{i_0}^{-1}\cdots A_{i_n}^{-1}\|^2}<C^{-1}r.
	$$
	Hence, $(e^{ss})_*\mu(B_r^{\sphericalangle}(e^{ss}(\ii)))\geq\mu([\ii|_0^{n(r,\ii)}])$. Therefore,
	\begin{equation}\label{eneedto}
	\frac{\log(e^{ss})_*\mu(B_r^{\sphericalangle}(e^{ss}(\ii)))}{\log r}\leq\frac{\log\mu([\ii|_0^{n(r,\ii)}])}{\log C+\log|\det(A_{i_0}^{-1}\cdots A_{i_{n(r,\ii)-1}}^{-1})|-2\log\|A_{i_0}^{-1}\cdots A_{i_{n(r,\ii)-1}}^{-1}\| }
	\end{equation}
	By ergodicity and Lemma~\ref{ldomsplit}\eqref{ldomsplit2},
	\begin{eqnarray*}
		&&\lim_{n\rightarrow\infty}-\frac{1}{n}\log\mu([\ii|_0^{n}])=h_{\mu}\\
		&&\lim_{n\rightarrow\infty}-\frac{1}{n}\log|\det(A_{i_0}^{-1}\cdots A_{i_{n-1}}^{-1})|=-\chi^{ss}_{\mu}-\chi^{s}_{\mu}\\
		&&\lim_{n\rightarrow\infty}\frac{1}{n}\log\|A_{i_0}^{-1}\cdots A_{i_{n-1}}^{-1}\|=\chi^{ss}_{\mu}\text{ for }\mu\text{-a.e. }\ii\in\Sigma^+.
	\end{eqnarray*}
    Putting these limits into \eqref{eneedto} completes the proof.
\end{proof}

\begin{lemma}\label{lcontsubspace}
	Let $U\subset\R^d$ be open and bounded and let $\mathcal{A}(\lv)=\left\{A_i(\lv)\right\}_{i=1}^N$ be a parametrized family of matrices such that the map $\lv\mapsto A_i(\lv)$ is continuous for any $i=1,\dots,N$ in an open neighbourhood of $\overline{U}$, and $\mathcal{A}(\lv)$ satisfies the dominated splitting on $\overline{U}$. Then the map $\lv\mapsto e^{ss}_{\lv}(\ii)$ is uniformly continuous for every $\ii\in\Sigma^+$. That is, for every $\lv_0\in U$ and every $\varepsilon>0$ there exists a $\delta=\delta(\lv_0,\varepsilon)>0$ that
	$$
	\|\lv-\lv_0\|<\delta\implies\sphericalangle(e_{\lv}^{ss}(\ii),e_{\lv_0}^{ss}(\ii))<\varepsilon\text{ for every $\ii\in\Sigma^+$.}
	$$
\end{lemma}

\begin{proof}
	Let $\lv_0\in U$ and $\varepsilon>0$ be arbitrary but fixed. Let $M$ be the backward invariant multicone of $\mathcal{A}(\lv_0)$. By definition of backward invariant multicone, there exists a $\delta'=\delta'(\lv_0)>0$ that for every $\lv$ with $\|\lv-\lv_0\|<\delta'$, $M$ is a backward invariant multicone for $\mathcal{A}(\lv)$. Hence, the angles between the directions of the dominated splitting are uniformly bounded from below. Thus, by Lemma~\ref{ldomsplit}\eqref{ldomsplit2} and Lemma~\ref{ldistcomp}, there exists a constant $C=C(\lv_0)>0$ that for every for every, $m\geq0$ integer we have
	\begin{multline*}
	\sphericalangle(e_{\lv}^{ss}(\ii),e_{\lv_0}^{ss}(\ii))\leq\\
	\sphericalangle(A_{i_0}^{-1}(\lv_0)\cdots A_{i_m}^{-1}(\lv_0)e_{\lv_0}^{ss}(\sigma^{m+1}\ii),A_{i_0}^{-1}(\lv_0)\cdots A_{i_m}^{-1}(\lv_0)e_{\lv}^{ss}(\sigma^{m+1}\ii))+\\
	\sphericalangle(A_{i_0}^{-1}(\lv_0)\cdots A_{i_m}^{-1}(\lv_0)e_{\lv}^{ss}(\sigma^{m+1}\ii),A_{i_0}^{-1}(\lv)\cdots A_{i_m}^{-1}(\lv)e_{\lv}^{ss}(\sigma^{m+1}\ii))\leq\\ C(\lv_0)^22\dfrac{|\det(A_{i_0}^{-1}(\lv_0)\cdots A_{i_m}^{-1}(\lv_0))|}{\|A_{i_0}^{-1}(\lv_0)\cdots A_{i_m}^{-1}(\lv_0)\|^2}\sphericalangle(e_{\lv}^{ss}(\sigma^{m+1}\ii),e_{\lv_0}^{ss}(\sigma^{m+1}\ii))+\\
	\dfrac{\sum_{i=1}^2|A_{i_0}^{-1}(\lv)\cdots A_{i_m}^{-1}(\lv)\underline{u}_i\times A_{i_0}^{-1}(\lv_0)\cdots A_{i_m}^{-1}(\lv_0)\underline{u}_i|+|\sum_{i=1}^2A_{i_0}^{-1}(\lv)\cdots A_{i_m}^{-1}(\lv)\underline{u}_i\times A_{i_0}^{-1}(\lv_0)\cdots A_{i_m}^{-1}(\lv_0)\underline{u}_{3-i}|}{\|A_{i_0}(\lv)\cdots A_{i_m}(\lv)\|^{-1}\|A_{i_0}(\lv_0)\cdots A_{i_m}(\lv_0)\|^{-1}},
	\end{multline*}
	where $\underline{u}_1$, $\underline{u}_2$ is the standard basis of $\R^2$. Since $\mathcal{A}(\lv)$ satisfies the dominated splitting on $\overline{U}$, there exists an integer $m=m(\lv_0)>0$ that
	$$
	C(\lv_0)^22\dfrac{|\det(A_{i_0}^{-1}(\lv_0)\cdots A_{i_m}^{-1}(\lv_0))|}{\|A_{i_0}^{-1}(\lv_0)\cdots A_{i_m}^{-1}(\lv_0)\|^2}<\frac{1}{2},
	$$
	for every $i_0,\dots,i_m\in\left\{1,\dots,N\right\}$. Let $f(\lv,\lv_0):=\sup_{\ii\in\Sigma^+}\sphericalangle(e_{\lv}^{ss}(\ii),e_{\lv_0}^{ss}(\ii))$, then
	\begin{multline*}
	f(\lv,\lv_0)\leq	
	2\max_{i_0,\dots,i_m}\left\{\dfrac{\sum_{i=1}^2|A_{i_0}^{-1}(\lv)\cdots A_{i_m}^{-1}(\lv)\underline{u}_i\times A_{i_0}^{-1}(\lv_0)\cdots A_{i_m}^{-1}(\lv_0)\underline{u}_i|}{\|A_{i_0}(\lv)\cdots A_{i_m}(\lv)\|^{-1}\|A_{i_0}(\lv_0)\cdots A_{i_m}(\lv_0)\|^{-1}}\right.+\\
	\left.\frac{|\sum_{i=1}^2A_{i_0}^{-1}(\lv)\cdots A_{i_m}^{-1}(\lv)\underline{u}_i\times A_{i_0}^{-1}(\lv_0)\cdots A_{i_m}^{-1}(\lv_0)\underline{u}_{3-i}|}{\|A_{i_0}(\lv)\cdots A_{i_m}(\lv)\|^{-1}\|A_{i_0}(\lv_0)\cdots A_{i_m}(\lv_0)\|^{-1}}\right\}.
	\end{multline*}
	Since the maps $\lv\mapsto A_i(\lv)$ are continuous, there exists a $\delta=\delta(\lv_0,\varepsilon)>0$ that the right hand side is less that $\varepsilon>0$ for every $\lv$ with $\|\lv-\lv_0\|<\delta$.
\end{proof}

\begin{lemma}\label{lcontlyap}
	Let $U\subset\R^d$ be open and bounded and let $\left\{\mu_{\lv}\right\}_{\lv\in U}$ be a family of $\sigma$-invariant quasi-Bernoulli ergodic measures on $\Sigma^+$ that \eqref{econtmeasure} holds. Then the map $\lv\mapsto\mu_{\lv}$ is continuous in weak*-topology. Moreover, if $\mathcal{A}(\lv)=\left\{A_i(\lv)\right\}_{i=1}^N$ is a parametrized family of matrices that the map $\lv\mapsto A_i(\lv)$ is continuous for any $i=1,\dots,N$ in an open neighbourhood of $\overline{U}$, and for every $\lv\in\overline{U}$ the set $\mathcal{A}(\lv)$ satisfies the dominated splitting then the maps $\lv\mapsto\chi_{\mu_{\lv}}^{ss}(\lv)$ and $\lv\mapsto\chi_{\mu_{\lv}}^{s}(\lv)$ are continuous.
\end{lemma}

\begin{proof}
	To prove the first assertion of the lemma it is enough to show that for every $O\subseteq\Sigma^+$ open set and every $\lv_0\in U$
	\begin{equation}\label{eweakenough}
	\liminf_{\lv\mapsto\lv_0}\mu_{\lv}(O)\geq\mu_{\lv_0}(O).
	\end{equation}
	Since the cylinder sets form a base of open sets we get $O=\bigcup_{k=1}^{\infty}[\ii_k|_{n_k} ^{m_k}]$. Since for every cylinder $[\ii_k|_{n_k} ^{m_k}]=\bigcup_{|\underline{j}|=n_k}[\underline{j}\sigma^{n_k}\ii_k|_{0} ^{m_k}]$ without loss of generality we may write $O=\bigcup_{k=1}^{\infty}[\ii_k|_{0} ^{m_k}]$. On the other hand, for every pair of cylinder sets of the form $[\ii_k|_{0} ^{m_k}]$ either they are disjoint or one contains the other, thus, we may assume that $[\ii_k|_{0} ^{m_k}]\cap[\ii_l|_{0} ^{m_l}]=\emptyset$ if $k\neq l$. Hence,
	\begin{equation*}
	\mu_{\lv_0}(O)=\lim_{n\rightarrow\infty}\sum_{\substack{|\underline{i}|=n \\ [\underline{i}]\subseteq O}}\mu_{\lv_0}([\underline{i}]).
	\end{equation*}
	Therefore, by \eqref{econtmeasure} for every $n\geq1$
	$$
	\liminf_{\lv\rightarrow\lv_0}\mu_{\lv}(O)\geq\liminf_{\lv\rightarrow\lv_0}\sum_{\substack{|\underline{i}|=n \\ [\underline{i}]\subseteq O}}\mu_{\lv}([\underline{i}])=\sum_{\substack{|\underline{i}|=n \\ [\underline{i}]\subseteq O}}\mu_{\lv_0}([\underline{i}]).
	$$
	Since $n\geq1$ was arbitrary we get \eqref{eweakenough}.
	
	To prove the second assertion, by Lemma~\ref{ldomsplit}\eqref{ldomsplit2} and multiplicative ergodic theorem
	$$
	\chi_{\mu_{\lv}}^{ss}(\lv)=\int\log\|A_{i_0}^{-1}(\lv)|e^{ss}_{\lv}(\sigma\ii)\|d\mu_{\lv}(\ii)\text{ and }	\chi_{\mu_{\lv}}^{ss}(\lv)+\chi_{\mu_{\lv}}^{s}(\lv)=\int\log|\det(A_{i_0}^{-1}(\lv))|d\mu_{\lv}(\ii).
	$$
	By Lemma~\ref{lcontsubspace}, the map $\lv\mapsto\log\|A_{i_0}^{-1}(\lv)|e^{ss}_{\lv}(\sigma\ii)\|$ is continuous, thus by the weak*-continuity of $\lv\mapsto\mu_{\lv}$, the map $\lv\mapsto\chi_{\mu}^{ss}(\lv)$ is continuous. The continuity of $\lv\mapsto\chi_{\mu_{\lv}}^{s}(\lv)$ follows by the continuity of $\lv\mapsto\mu_{\lv}$, $\lv\mapsto\chi_{\mu_{\lv}}^{ss}(\lv)$ and $\lv\mapsto\log|\det(A_{i_0}^{-1}(\lv))|$.
\end{proof}

\begin{prop}\label{pdimlow}
	Assume that the assumptions of Theorem~\ref{tss} hold. Then for every $\lv_0\in U$ and $\varepsilon>0$ there exists a $\delta>0$ such that
	$$
	\dim_H(e^{ss}_{\lv})_*\mu_{\lv}\geq\min\left\{1,\frac{h_{\mu_{\lv_0}}}{\chi_{\mu_{\lv_0}}^{ss}(\lv_0)-\chi_{\mu_{\lv_0}}^{s}(\lv_0)}\right\}-\varepsilon\text{ for $\mathcal{L}_d$-a.e. $\lv\in B_{\delta}(\lv_0)$}.
	$$
\end{prop}

Before we prove Proposition~\ref{pdimlow}, we prove that for every $\lv\in U$ the map $\ii\mapsto e_{\lv}^{ss}(\ii)$ is H\"older continuous.

\begin{lemma}\label{ltechbound}
	For every $\lv_0\in U$ there exists a $\delta=\delta(\lv_0)>0$ and for every $r>0$ there exists a positive integer $N=N(\lv_0,r)$ that for every $\lv\in U$ with $\|\lv-\lv_0\|<\delta$ and for every $\ii,\jj\in\Sigma^+$ with $i_0\neq j_0$
	$$
	\mathbb{I}\left\{\sphericalangle(e^{ss}_{\lv}(\ii),e^{ss}_{\lv}(\jj))<r\right\}\leq\mathbb{I}\left\{\sphericalangle(e^{ss}_{\lv}(\ii|_0^N\overline{1}),e^{ss}_{\lv}(\jj|_0^N\overline{1}))<2r\right\},
	$$
	where $\overline{1}=(1,1,\dots)\in\Sigma^+$ and $\mathbb{I}$ denotes the indicator function. Precisely, $N(\lv_0,r)=\lceil\frac{2\log r}{-\beta(\lv_0)}+c(\lv_0)\rceil$, where $\beta(\lv_0)$ is the domination exponent in Definition~\ref{ddomsplit} and $c(\lv_0)$ is some constant depending only on $\lv_0$.
\end{lemma}

\begin{proof}
	Fix $\lv_0\in U$. Then by Lemma~\ref{ldistcomp} for every $N$ and every $\ii,\jj\in\Sigma^+$ with $i_0\neq j_0$
	\begin{multline*}
	|\sphericalangle(e^{ss}_{\lv}(\ii),e^{ss}_{\lv}(\jj))-\sphericalangle(e^{ss}_{\lv}(\ii|_0^N\overline{1}),e^{ss}_{\lv}(\jj|_0^N\overline{1}))|\leq\sphericalangle(e^{ss}_{\lv}(\ii),e^{ss}_{\lv}(\ii|_0^N\overline{1}))+\sphericalangle(e^{ss}_{\lv}(\jj),e^{ss}_{\lv}(\jj|_0^N\overline{1}))\leq\\
	2\frac{|\det(A_{i_0}^{-1}(\lv)\cdots A_{i_{N}}^{-1}(\lv))|}{\|A_{j_0}^{-1}(\lv)\cdots A_{j_{N}}^{-1}(\lv)|e^{ss}_{\lv}(\sigma^{N+1}\ii)\|\|A_{j_0}^{-1}(\lv)\cdots A_{j_{N}}^{-1}(\lv)|e^{ss}_{\lv}(\overline{1})\|}\sphericalangle(e^{ss}_{\lv}(\sigma^{N+1}\ii),e^{ss}_{\lv}(\overline{1}))+\\
	2\frac{|\det(A_{j_0}^{-1}(\lv)\cdots A_{j_{N}}^{-1}(\lv))|}{\|A_{j_0}^{-1}(\lv)\cdots A_{j_{N}}^{-1}(\lv)|e^{ss}_{\lv}(\sigma^{N+1}\jj)\|\|A_{j_0}^{-1}(\lv)\cdots A_{j_{N}}^{-1}(\lv)|e^{ss}_{\lv}(\overline{1})\|}\sphericalangle(e^{ss}_{\lv}(\sigma^{N+1}\jj),e^{ss}_{\lv}(\overline{1})).
	\end{multline*}
	Since $\lv\mapsto A_i(\lv)$ is continuous, by Lemma~\ref{lcontsubspace}, there exists a $\delta=\delta(\lv_0)>0$ that
	\begin{multline*}
	\frac{|\det(A_{j_0}^{-1}(\lv)\cdots A_{j_{N}}^{-1}(\lv))|}{\|A_{j_0}^{-1}(\lv)\cdots A_{j_{N}}^{-1}(\lv)|e^{ss}_{\lv}(\sigma^{N+1}\jj)\|\|A_{j_0}^{-1}(\lv)\cdots A_{j_{N}}^{-1}(\lv)|e^{ss}_{\lv}(\overline{1})\|}\leq\\
	e^{\frac{\delta(\lv_0)}{2} N}\frac{|\det(A_{j_0}^{-1}(\lv_0)\cdots A_{j_{N}}^{-1}(\lv_0))|}{\|A_{j_0}^{-1}(\lv_0)\cdots A_{j_{N}}^{-1}(\lv_0)|e^{ss}_{\lv_0}(\sigma^{N+1}\jj)\|\|A_{j_0}^{-1}(\lv_0)\cdots A_{j_{N}}^{-1}(\lv_0)|e^{ss}_{\lv_0}(\overline{1})\|}
	\end{multline*}
	for every $\jj\in\Sigma^+$. Thus, by Lemma~\ref{ldomsplit}\eqref{ldomsplit2}
	\begin{multline*}
	|\sphericalangle(e^{ss}_{\lv}(\ii),e^{ss}_{\lv}(\jj))-\sphericalangle(e^{ss}_{\lv}(\ii|_0^N\overline{1}),e^{ss}_{\lv}(\jj|_0^N\overline{1}))|\leq
	2\pi e^{\frac{\delta(\lv_0)}{2} N}C(\lv_0)^2\max_{j_0,\dots,j_N}\left\{\frac{|\det(A_{j_0}^{-1}(\lv_0)\cdots A_{j_{N}}^{-1}(\lv_0))|}{\|A_{j_0}^{-1}(\lv_0)\cdots A_{j_{N}}^{-1}(\lv_0)\|^2}\right\}
	\end{multline*}
	By Definition~\ref{ddomsplit}, there exists an $N=N(\lv_0,r)$ that the right hand side of the inequality is less than $r$, thus the statement follows.
\end{proof}

\begin{proof}[Proof of Proposition~\ref{pdimlow}]
	Let $\lv_0\in U$  and $\varepsilon>0$ be arbitrary but fixed. Let $\delta=\delta(\lv_0,\varepsilon)>0$ be chosen according to Lemma~\ref{lcontsubspace}, Lemma~\ref{ltechbound} and \eqref{econtmeasure}. By Shannon-McMillan-Breiman Theorem and \eqref{econtmeasure}, for every $\lv\in B_{\delta}(\lv_0)$
	\begin{equation*}
	h_{\mu_{\lv_0}}-\varepsilon\leq\liminf_{n\rightarrow\infty}-\frac{1}{n}\log\mu_{\lv}([\ii|_0^{n-1}])\leq\limsup_{n\rightarrow\infty}-\frac{1}{n}\log\mu_{\lv}([\ii|_0^{n-1}])
	\leq h_{\mu_{\lv_0}}+\varepsilon\text{ for $\mu_{\lv}$-a.e. $\ii\in\Sigma^+$}.
	\end{equation*}
	Moreover, by ergodic theorem and weak*-continuity of $\lv\mapsto\mu_{\lv}$
	\begin{equation*}
	\chi^{ss}_{\mu_{\lv_0}}(\lv_0)+\chi^{s}_{\mu_{\lv_0}}(\lv_0)-\varepsilon\leq\lim_{n\rightarrow\infty}\frac{1}{n}\log|\det(A_{i_0}^{-1}(\lv)\cdots A_{i_{n-1}}^{-1}(\lv))|\leq\chi^{ss}_{\mu_{\lv_0}}(\lv_0)+\chi^{s}_{\mu_{\lv_0}}(\lv_0)+\varepsilon,
	\end{equation*}
	\begin{equation*}
	\chi^{ss}_{\mu_{\lv_0}}(\lv_0)-\varepsilon\leq\lim_{n\rightarrow\infty}\frac{1}{n}\log\|A_{i_0}^{-1}(\lv)\cdots A_{i_{n-1}}^{-1}(\lv)|e^{ss}_{\lv}(\sigma^n\ii)\|\leq \chi^{ss}_{\mu_{\lv_0}}(\lv_0)+\varepsilon
	\end{equation*}
	for $\mu_{\lv}$-a.e. $\ii\in\Sigma^+$. By Egorov's theorem for every $\lv\in B_{\delta}(\lv_0)$ there exists a set $\Omega_{\lv}\subseteq\Sigma^+$ that $\mu(\Omega_{\lv})>1-\varepsilon$ and there exist a constant $C(\lv)>1$ that for every $\ii\in\Sigma^+$ and every $n,m\geq1$
	\begin{equation*}
	C(\lv)^{-1}\mu_{\lv}([\ii|_0^{n-1}])\mu_{\lv}([\sigma^n\ii|_0^{m-1}])\leq\mu_{\lv}([\ii|_0^{n+m-1}])\leq C(\lv)\mu_{\lv}([\ii|_0^{n-1}])\mu_{\lv}([\sigma^n\ii|_0^{m-1}])	
	\end{equation*}
	and for every $\ii\in\Omega_{\lv}$ and every $n\geq1$
	\begin{equation}\label{eentaprox}
	C(\lv)^{-1}e^{-n(h_{\mu_{\lv_0}}+2\varepsilon)}\leq\mu_{\lv}([\ii|_0^{n-1}])\leq C(\lv)e^{-n(h_{\mu_{\lv_0}}-2\varepsilon)},
	\end{equation}
	\begin{equation}\label{elyapexp}
	C(\lv)^{-1}e^{-n(\chi_{\mu_{\lv_0}}^{ss}(\lv_0)-\chi_{\mu_{\lv_0}}^{s}(\lv_0)+6\varepsilon)}\leq\frac{|\det(A_{i_0}^{-1}(\lv)\cdots A_{i_{n-1}}^{-1}(\lv))|}{\|A_{i_0}^{-1}(\lv)\cdots A_{i_{n-1}}^{-1}(\lv)\|^2}\leq C(\lv)e^{-n(\chi_{\mu_{\lv_0}}^{ss}(\lv_0)-\chi_{\mu_{\lv_0}}^{s}(\lv_0)-6\varepsilon)}.
	\end{equation}
	By Lusin's theorem for every $\varepsilon'>0$ there exists a set $J_{\delta}(\lv_0)\subseteq B_{\delta}(\lv_0)$ that $\mathcal{L}_d(B_{\delta}(\lv_0)/J_{\delta}(\lv_0))<\varepsilon'$ and there exists a $C>1$ that $C(\lv)\leq C$ for every $\lv\in J_{\delta}(\lv_0)$. Denote the measure $\widetilde{\mu}_{\lv}:=\left.\mu\right|_{\Omega_{\lv}}$ and for a finite length word $\underline{k}=(k_0,\dots,k_{n-1})$ denote the set
	$$
	\Sigma_{\underline{k}}:=\left\{(\ii,\jj)\in\Sigma^+:i_m=j_m=k_m\text{ for }m=0,\dots,n-1\text{ and }i_{n}\neq j_{n}\right\}.
	$$
	Then for every $s>0$ by Lemma~\ref{lcontsubspace}, the continuity of $\lv\mapsto A_i(\lv)$ and \eqref{elyapexp}
	\begin{multline*}
	\mathcal{I}:=\int_{J_{\delta}(\lv_0)}\iint\sphericalangle(e^{ss}_{\lv}(\ii)),e^{ss}_{\lv}(\jj))^{-s}d\widetilde{\mu}_{\lv}(\ii)d\widetilde{\mu}_{\lv}(\jj)d\lv=\\
	\sum_{n=0}^{\infty}\sum_{\underline{k}=n}\int_{J_{\delta}(\lv_0)}\iint_{\Sigma_{\underline{k}}}\sphericalangle(e^{ss}_{\lv}(\ii)),e^{ss}_{\lv}(\jj))^{-s}d\widetilde{\mu}_{\lv}(\ii)d\widetilde{\mu}_{\lv}(\jj)d\lv\leq\\
	\sum_{n=0}^{\infty}\sum_{\underline{k}=n}\int_{J_{\delta}(\lv_0)}\iint_{\Sigma_{\underline{k}}}C(\lv)^2\left(\dfrac{|\det(A_{k_0}^{-1}(\lv)\cdots A_{k_{n-1}}^{-1}(\lv))|}{2\|A_{k_0}^{-1}(\lv)\cdots A_{k_{n-1}}^{-1}(\lv)\|^2}\right)^{-s}\sphericalangle(e^{ss}_{\lv}(\sigma^n\ii)),e^{ss}_{\lv}(\sigma^n\jj))^{-s}d\widetilde{\mu}_{\lv}(\ii)d\widetilde{\mu}_{\lv}(\jj)d\lv\leq\\
	\sum_{n=0}^{\infty}C'e^{sn(\chi_{\mu}^{ss}(\lv_0)-\chi_{\mu}^{s}(\lv_0)+6\varepsilon)}\sum_{\underline{k}=n}\int_{J_{\delta}(\lv_0)}\iint_{\Sigma_{\underline{k}}}\sphericalangle(e^{ss}_{\lv}(\sigma^n\ii)),e^{ss}_{\lv}(\sigma^n\jj))^{-s}d\widetilde{\mu}_{\lv}(\ii)d\widetilde{\mu}_{\lv}(\jj)d\lv.
	\end{multline*}
	By Lemma~\ref{ltechbound}, for any $\underline{k}$ with $|\underline{k}|=n$
	\begin{multline}\label{eneedcont}
	\mathcal{I}_{\underline{k}}:=\int_{J_{\delta}(\lv_0)}\iint_{\Sigma_{\underline{k}}}\sphericalangle(e^{ss}_{\lv}(\sigma^n\ii)),e^{ss}_{\lv}(\sigma^n\jj))^{-s}d\widetilde{\mu}_{\lv}(\ii)d\widetilde{\mu}_{\lv}(\jj)d\lv\leq\\
	\sum_{m=0}^{\infty}2^{(m+1)s}	\int_{J_{\delta}(\lv_0)}\iint_{\Sigma_{\underline{k}}}\mathbb{I}\left\{\sphericalangle(e^{ss}_{\lv}(\sigma^n\ii),e^{ss}_{\lv}(\sigma^n\jj))<\frac{1}{2^m}\right\}d\widetilde{\mu}_{\lv}(\ii)d\widetilde{\mu}_{\lv}(\jj)d\lv\leq\\
	\sum_{m=0}^{\infty}2^{(m+1)s}\int_{J_{\delta}(\lv_0)}\iint_{\Sigma_{\underline{k}}}\mathbb{I}\left\{\sphericalangle(e^{ss}_{\lv}(\sigma^n\ii|_0^{N(\lv_0,m)}\overline{1}),e^{ss}_{\lv}(\sigma^n\jj|_0^{N(\lv_0,m)}\overline{1}))<\frac{2}{2^m}\right\}d\widetilde{\mu}_{\lv}(\ii)d\widetilde{\mu}_{\lv}(\jj)d\lv=\\
	\sum_{m=0}^{\infty}2^{(m+1)s}
	\sum_{\substack{|\underline{l}|=N(\lv_0,m) \\ |\underline{h}|=N(\lv_0,m)}}\int_{J_{\delta}(\lv_0)}\iint_{[\underline{k}\underline{l}]\times[\underline{k}\underline{h}]}\mathbb{I}\left\{\sphericalangle(e^{ss}_{\lv}(\underline{h}\overline{1}),e^{ss}_{\lv}(\underline{l}\overline{1}))<\frac{2}{2^m}\right\}d\widetilde{\mu}_{\lv}(\ii)d\widetilde{\mu}_{\lv}(\jj)d\lv\\
	\end{multline}
	By applying \eqref{eentaprox}, the quasi-Bernoulli property of $\mu_{\lv_0}$, \eqref{econtmeasure} and the continuity of $\lv\mapsto h_{\mu_{\lv}}$
	\begin{multline*}
	\int_{J_{\delta}(\lv_0)}\mathbb{I}\left\{\sphericalangle(e^{ss}_{\lv}(\underline{h}\overline{1}),e^{ss}_{\lv}(\underline{l}\overline{1}))<\frac{2}{2^m}\right\}\widetilde{\mu}_{\lv}([\underline{k}\underline{l}])\widetilde{\mu}_{\lv}([\underline{k}\underline{h}])d\lv\leq\\
	C^2\int_{J_{\delta}(\lv_0)}\mathbb{I}\left\{\sphericalangle(e^{ss}_{\lv}(\underline{h}\overline{1}),e^{ss}_{\lv}(\underline{l}\overline{1}))<\frac{2}{2^m}\right\}\widetilde{\mu}_{\lv}([\underline{k}])^2\widetilde{\mu}_{\lv}([\underline{l}])\widetilde{\mu}_{\lv}([\underline{h}])d\lv\leq\\
	c'\mu_{\lv_0}([\underline{k}])\mu_{\lv_0}([\underline{l}])\mu_{\lv_0}([\underline{h}])e^{2
		\varepsilon(n+N(\lv_0,m))}e^{-n(h_{\mu_{\lv_0}}-2\varepsilon)}\mathcal{L}_d\left(\lv\in J_{\delta}(\lv_0):\sphericalangle(e^{ss}_{\lv}(\underline{h}\overline{1}),e^{ss}_{\lv}(\underline{l}\overline{1}))<\frac{2}{2^m}\right).
	\end{multline*}
	Hence, by \eqref{eneedcont} and the strong-stable transversality
	\begin{multline*}
	\mathcal{I}_{\underline{k}}\leq
	c'\mu_{\lv_0}([\underline{k}])\sum_{m=0}^{\infty}2^{(m+1)s}
	\sum_{\substack{|\underline{l}|=N(\lv_0,m) \\ |\underline{h}|=N(\lv_0,m)}}\mu_{\lv_0}([\underline{l}])\mu_{\lv_0}([\underline{l}])e^{2
		\varepsilon(n+N(\lv_0,m))}e^{-n(h_{\mu_{\lv_0}}-2\varepsilon)}\frac{C}{2^m}=\\
	c''\mu_{\lv_0}([\underline{k}])e^{-n(h_{\mu_{\lv_0}}-4\varepsilon)}\sum_{m=0}^{\infty}2^{m(s-1)+2\varepsilon N(\lv_0,m)/\log2}
	\end{multline*}
	Since $N(\lv_0,m)/\log2\leq m\frac{2}{\beta(\lambda_0)}+c(\lv_0)$
	\begin{equation*}
	\mathcal{I}\leq c'''\sum_{n=0}^{\infty}e^{n(s(\chi_{\mu}^{ss}(\lv_0)-\chi_{\mu}^{s}(\lv_0))-h_{\mu_{\lv_0}}+10\varepsilon)}\sum_{m=0}^{\infty}2^{m(s-1+\varepsilon\frac{4}{\beta(\lambda_0)})}.
	\end{equation*}
	Hence, by choosing $s<\min\left\{1-\varepsilon\frac{5}{\beta(\lv_0)},\frac{h_{\mu_{\lv_0}}-11\varepsilon}{\chi_{\mu}^{ss}(\lv_0)-\chi_{\mu}^{s}(\lv_0)}\right\}$ the right hand side of the inequality is finite. By Frostman's Lemma \cite[Theorem 4.13]{Fb1},
	$$
	\dim_H(e^{ss}_{\lv})_*\widetilde{\mu}_{\lv}\geq\min\left\{1-\varepsilon\frac{5}{\beta(\lv_0)},\frac{h_{\mu_{\lv_0}}-11\varepsilon}{\chi_{\mu}^{ss}(\lv_0)-\chi_{\mu}^{s}(\lv_0)}\right\}\text{ for $\mathcal{L}_d$-a.e. $\lv\in J_{\delta}(\lv_0)$.}
	$$
	But for every $\lv\in B_{\delta}(\lv_0)$, $\dim_H(e^{ss}_{\lv})_*\mu_{\lv}\geq\dim_H(e^{ss}_{\lv})_*\widetilde{\mu}_{\lv}$, moreover, $\mathcal{L}_d(B_{\delta}(\lv_0)/J_{\delta}(\lv_0))$ can be chosen arbitrary small, thus, the statement follows.
\end{proof}

\begin{proof}[Proof of Theorem~\ref{tss}]
	By Lemma~\ref{lubss} we have
	$$
	\dim_H(e^{ss}_{\lv})_*\mu_{\lv}\leq\min\left\{1,\frac{h_{\mu_{\lv}}}{\chi_{\mu_{\lv}}^{ss}(\lv)-\chi_{\mu_{\lv}}^{s}(\lv)}\right\}\text{ for every $\lv\in U$}.
	$$
	So it is enough to establish the lower bound. Let us argue by contradiction. Assume that there exist a set $U'\subset U$ with $\mathcal{L}_d(U')>0$ such that
	$$
	\dim_H(e^{ss}_{\lv})_*\mu_{\lv}\leq\min\left\{1,\frac{h_{\mu_{\lv}}}{\chi_{\mu_{\lv}}^{ss}(\lv)-\chi_{\mu_{\lv}}^{s}(\lv)}\right\}-\varepsilon \text{ for $\mathcal{L}_d$-a.e. $\lv\in U'$ for some $\varepsilon>0$}.
	$$
	Let $\lv_0\in U'$ a Lebesgue density point. Thus, there exists a $\delta_0>0$ that for every $\delta_0>\delta>0$
	$$\mathcal{L}_d\left(\lv\in B_{\delta}(\lv_0):	\dim_H(e^{ss}_{\lv})_*\mu_{\lv}\leq\min\left\{1,\frac{h_{\mu_{\lv}}}{\chi_{\mu_{\lv}}^{ss}(\lv)-\chi_{\mu_{\lv}}^{s}(\lv)}\right\}-\varepsilon\right)>0.$$
	By using the continuity of entropy and Lyapunov exponents we have for sufficiently small $\delta>0$
	$$\mathcal{L}_d\left(\lv\in B_{\delta}(\lv_0):	\dim_H(e^{ss}_{\lv})_*\mu_{\lv}\leq\min\left\{1,\frac{h_{\mu_{\lv_0}}}{\chi_{\mu_{\lv_0}}^{ss}(\lv_0)-\chi_{\mu_{\lv_0}}^{s}(\lv_0)}\right\}-\frac{\varepsilon}{2}\right)>0,$$
	but this contradicts Proposition~\ref{pdimlow}.
\end{proof}

\begin{proof}[Proof of Theorem~\ref{tdimGibbs2}]
    By \cite[Section~1]{Bbook}, a family of Gibbs measures for a uniformly continuously parametrized family of Holder continuous potentials is weakly continuous. Hence,  $\left\{\mu_{\lv}\right\}_{\lv\in U}$ satisfy equation~\eqref{econtmeasure}. Then by Theorem~\ref{tss}, we have
	$$
	\dim_H(\e{\lv}{ss})_*\mu_{\lv}=\min\left\{\frac{h_{\mu_{\lv}}}{\chi^{ss}_{\mu_{\lv}}(\lv)-\chi^{s}_{\mu_{\lv}}(\lv)},1\right\}\text{ for $\mathcal{L}_d$-a.e $\lv\in U$.}
	$$
	On the other hand, by Theorem~\ref{tdimGibbs}, if
	\begin{equation*}
	\frac{h_{\mu_{\lv}}}{\chi^{ss}_{\mu_{\lv}}(\lv)-\chi^{s}_{\mu_{\lv}}(\lv)}\geq\min\left\{\frac{h_{\mu_{\lv}}}{\chi^{s}_{\mu_{\lv}}(\lv)},1\right\}
	\end{equation*}
	the statement holds. Thus, we may assume that
	$$
	\frac{h_{\mu_{\lv}}}{\chi^{ss}_{\mu_{\lv}}(\lv)-\chi^{s}_{\mu_{\lv}}(\lv)}<1\text{,\ }\chi^{ss}_{\mu_{\lv}}(\lv)>2\chi^{s}_{\mu_{\lv}}(\lv)\text{\ and\ }
	\frac{h_{\mu_{\lv}}}{\chi^{ss}_{\mu_{\lv}}(\lv)-\chi^{s}_{\mu_{\lv}}(\lv)}+2\frac{h_{\mu_{\lv}}}{\chi^{ss}_{\mu_{\lv}}(\lv)}>2.
	$$
	By \cite[Lemma~4.12]{B}, we get that $\dim_H(\pi^-_{\lv})_*\mu_{\lv}\geq2\frac{h_{\mu_{\lv}}}{\chi^{ss}_{\mu_{\lv}}(\lv)}$ and the statement follows by Theorem~\ref{tdimGibbs}.

\end{proof}

\section{Proof of Theorem~\ref{tmaindim}}

Finally, in this section we prove Theorem~\ref{tmaindim} as an application of Theorem~\ref{tdimGibbs2}.

For a matrix $A\in\R^{2\times2}_+\cup\R^{2\times2}_-$ let
	\begin{equation}\label{eredifs}
	S(x,A):=\frac{|a|x+|c|(1-x)}{(|a|+|b|)x+(|c|+|d|)(1-x)}
	\text{ where }
	A=\left[\begin{matrix}
	a & b \\
	c & d
	\end{matrix}\right].
	\end{equation}
 Simple calculations show that the maps $S_i\in C^2[0,1]$, Moreover,
 \begin{equation}\label{econtr}
 \sup_{x\in[0,1]}|S'(x,A)|=\max\left\{|S'(0,A)|,|S'(1,A)|\right\}=\dfrac{|\det A|}{\vvvert A\vvvert^2},\text{ and}
 \end{equation}
 \begin{equation*}
 \inf_{x\in[0,1]}|S'(x,A)|=\min\left\{|S'(0,A)|,|S'(1,A)|\right\}=\dfrac{|\det A|}{\|A\|_{\infty}^2},
 \end{equation*}
 where $\|A\|_{\infty}=\max\left\{|a|+|b|,|c|+|d|\right\}$ the usual $\infty$-norm of matrices.

\begin{lemma}
	Let $\mathcal{A}=\left\{A_1,\dots,A_N\right\}$ be a set of non-singular matrices with either strictly positive or strictly negative elements such that $\frac{|\det A_i|}{\vvvert A_i\vvvert^2}<1$. Let $\phi=\left\{S_i(.):=S(.,A_i)\right\}_{i=1}^N$ be IFS on $[0,1]$ and let $\Pi:\Sigma^+\mapsto[0,1]$ be the natural projection of $\phi$. Then for every $\ii_+\in\Sigma^+$ the vector~$\left(\Pi(\ii_+)-1,\Pi(\ii_+)\right)^T\in e^{ss}(\ii_+).$
\end{lemma}

\begin{proof}
	Let $\mathcal{A}=\left\{A_1,\dots,A_N\right\}$ and the IFS~$\phi=\left\{S_1,\dots,S_N\right\}$ be as required. It is easy to see that the cone $M=\left\{(x,y)\in\R^2/\left\{(0,0)\right\}:xy\leq0\right\}$ is backward invariant. So, by \cite[Theorem~B]{BG}, $\mathcal{A}$ satisfies the dominated splitting.

	For an $\ii_+\in\Sigma^+$ let $e^{ss}(\ii_+)$ be the invariant strong stable direction defined in \eqref{eredifs}. By the definition of $\Pi:\Sigma^+\mapsto[0,1]$
	\begin{multline*}
	\left(\begin{matrix}\Pi(\ii_+)-1\\ \Pi(\ii_+)\end{matrix}\right)=\dfrac{\left(\begin{matrix}
		-b_{i_0}\Pi(\sigma\ii_+)-d_{i_0}(1-\Pi(\sigma\ii_+)) \\
		a_{i_0}\Pi(\sigma\ii_+)+c_{i_0}(1-\Pi(\sigma\ii_+))
		\end{matrix}\right)}{(|a_{i_0}|+|b_{i_0}|)\Pi(\sigma\ii_+)+(|c_{i_0}|+|d_{i_0}|)(1-\Pi(\sigma\ii_+))}=\\
	\dfrac{\det A_{i_0}}{(|a_{i_0}|+|b_{i_0}|)\Pi(\sigma\ii_+)+(|c_{i_0}|+|d_{i_0}|)(1-\Pi(\sigma\ii_+))}A_{i_0}^{-1}\left(\begin{matrix}\Pi(\sigma\ii_+)-1\\ \Pi(\sigma\ii_+)\end{matrix}\right).
	\end{multline*}
	Thus, by Lemma~\ref{ldomsplit} and uniqueness, the $1$ dimensional subspace $e^{ss}(\ii_+)$ contains $\left(\Pi(\ii_+)-1,\Pi(\ii_+)\right)^T$.
\end{proof}

\begin{lemma}\label{ldim}
	Let $\mathcal{A}=\left\{A_1,\dots,A_N\right\}$ be arbitrary such that $A_i\in\mathfrak{M}$, where $\mathfrak{M}$ is defined in \eqref{emxset}. Moreover, let $\mathcal{A}(\underline{t})=\left\{A_1+t_1B_1,\dots,A_N+t_NB_N\right\}$, where $\underline{t}\in\R^N$
	\begin{equation}\label{edefb}
	A_i=\left(\begin{matrix}
	a_i & b_i \\
	c_i & d_i
	\end{matrix}\right)\text{ and }B_i=\left(\begin{matrix}
	a_i+b_i & -(a_i+b_i) \\
	c_i+d_i & -(c_i+d_i)
	\end{matrix}\right).
	\end{equation}
	Then there exists a $\delta=\delta(\mathcal{A})>0$ such that the IFS $\phi_{\tv}=\left\{S_i^{\tv}(.):=S(.,A_i+t_iB_i)\right\}_{i=1}^N$ satisfies the transversality condition on $(-\delta,\delta)^N$.
	
	In particular, $\mathcal{A}(\underline{t})$ satisfies the strong-stable transversality condition on $(-\delta,\delta)^N$.
\end{lemma}

\begin{proof}
	Since $\mathfrak{M}^N$ is open, there exists a $\varepsilon=\varepsilon(\mathcal{A})>0$ that $\mathcal{A}(\underline{t})\in\mathfrak{M}^N$ for every $\underline{t}\in(-\varepsilon,\varepsilon)^N$. Let $\phi=\left\{S_1,\dots,S_N\right\}$ be the IFS for $\mathcal{A}$ and $\phi_{\underline{t}}=\left\{S_1^{\underline{t}},\dots,S_N^{\underline{t}}\right\}$ be the IFS for $\mathcal{A}(\underline{t})$. Simple calculations show that $S_i^{\underline{t}}(x)=S_i(x)+t_i$ for every $i=1,\dots,N$. By the definition of $\mathfrak{M}$, by \eqref{econtr} and by \cite[Corollary~7.3]{SSU} there exists $\delta=\delta(\mathcal{A})>0$ such that $\delta<\varepsilon$ and $\phi_{\underline{t}}$ satisfies the transversality condition. By Lemma~\ref{ldim} and Definition~\ref{dsstrans}, it follows that $\mathcal{A}(\underline{t})$ satisfies the strong-stable transversality on $(-\delta,\delta)^N$.
\end{proof}

\begin{lemma}\label{lfoli}
	Let us define for every $\mathcal{A}\in \mathfrak{M}^N$
	\begin{equation*}
	P(\mathcal{A}):=\mathfrak{M}^N\cap\bigcup_{\underline{t}\in\R^N}\mathcal{A}(\underline{t}),
	\end{equation*}
	where $\mathcal{A}(\underline{t})$ is defined in Lemma~\ref{ldim}. Then $P$ defines a measurable partition of $\mathfrak{M}^N$.
\end{lemma}

\begin{proof}
	By the definition of $P$ it is enough to show that if $\mathcal{A}\neq\mathcal{A}'$ then either $P(\mathcal{A})=P(\mathcal{A}')$ or $P(\mathcal{A})\cap P(\mathcal{A}')=\emptyset$.
	
	Let us fix $\mathcal{A}\neq\mathcal{A}'$ and suppose that $P(\mathcal{A})\cap P(\mathcal{A}')\neq\emptyset$. Then there exist $t_1,\dots,t_N\in\R$ and $t_1',\dots,t_N'\in\R$ that $A_i+t_iB_i=A_i'+t_i'B_i'$ for every $i=1,\dots,N$, where $B_i$ and $B_i'$ defined in \eqref{edefb}. Thus $a_i+b_i=a_i'+b_i'$ and $c_i+d_i=c_i'+d_i'$. Hence, $P(\mathcal{A})=P(\mathcal{A}')$. The measurability is straightforward.
\end{proof}

\begin{proof}[Proof of Theorem~\ref{tmaindim}]
First we show that if $\mathcal{A}\in\mathfrak{N}^N\cup\mathfrak{O}_N$, where $\mathfrak{N}^N$ and $\mathfrak{O}_N$ are defined in \eqref{emxset2}, then condition \eqref{ctdim32} of Theorem~\ref{tdimGibbs2} holds for the K\"aenm\"aki measure $\mu^K$ of $\mathcal{A}$, defined in Definition~\ref{dKaenmaki}.
	
Indeed, if $\mathcal{A}\in\mathfrak{N}^N$ then $\dfrac{h_{\mu^K}}{\chi^{ss}_{\mu^K}-\chi^{s}_{\mu^K}}\geq\dfrac{h_{\mu^K}}{\chi^{s}_{\mu^K}}$ and on the other hand, if $\mathcal{A}\in\mathfrak{O}_N$ then
\begin{multline*}
\dfrac{h_{\mu^K}}{\chi^{ss}_{\mu^K}-\chi^{s}_{\mu^K}}+2\dfrac{h_{\mu^K}}{\chi^{ss}_{\mu^K}}=\dfrac{\chi^{s}_{\mu^K}+(s_0-1)\chi^{ss}_{\mu^K}}{\chi^{ss}_{\mu^K}-\chi^{s}_{\mu^K}}+2\dfrac{\chi^{s}_{\mu^K}+(s_0-1)\chi^{ss}_{\mu^K}}{\chi^{ss}_{\mu^K}}=\\
-3+\left(2+\dfrac{1}{1-\frac{\chi^{s}_{\mu^K}}{\chi^{ss}_{\mu^K}}}\right)s_0+2\frac{\chi^{s}_{\mu^K}}{\chi^{ss}_{\mu^K}}>\frac{1}{3}+\dfrac{5}{3\left(1-\frac{\chi^{s}_{\mu^K}}{\chi^{ss}_{\mu^K}}\right)}+2\frac{\chi^{s}_{\mu^K}}{\chi^{ss}_{\mu^K}}>2.
\end{multline*}

Now, let $V\subset\mathfrak{N}^N\cup\mathfrak{O}_N\subset\mathfrak{M}^N$ be a compact set such that $\overline{V^{o}}=V$. Let us define for a $\mathcal{A}\in V$
$$
Q(\mathcal{A}):=V\cap P(\mathcal{A}),
$$
Thus,  $\bigcup_{\mathcal{B}\in P(\mathcal{A})}\left\{\bigcup_{\underline{t}\in(-\delta(\mathcal{B}),\delta(\mathcal{B}))^N}\mathcal{B}(\underline{t})\right\}$ defines an open cover of $Q(\mathcal{A})$. Since $Q(\mathcal{A})$ is compact there is a finite set $\left\{\mathcal{B}_1,\dots,\mathcal{B}_n\right\}$ that $\bigcup_{i=1}^n\left\{\bigcup_{\underline{t}\in(-\delta(\mathcal{B}_i),\delta(\mathcal{B}_i))^N}\mathcal{B}_i(\underline{t})\right\}$ is a cover for $Q(\mathcal{A})$. But by Lemma~\ref{ldim}, for every $i=1,\dots,n$ the parametrized family of matrices $\mathcal{B}_i(\underline{t})$ satisfies the strong-stable transversality condition on $(-\delta(\mathcal{B}_i),\delta(\mathcal{B}_i))^N$. Thus, by Theorem~\ref{tdimGibbs2} for every $i=1,\dots,n$
$$
\dim_H\mu^K_{\tv}=\dim_H\Lambda_{\tv}=\dim_B\Lambda_{\tv}=s_0(\tv)\text{ for $\mathcal{L}_N$-a.e $\tv\in(-\delta(\mathcal{B}_i),\delta(\mathcal{B}_i))^N$,}
$$
where $\mu^K_{\tv}$ is the K\"aenm\"aki measure of the system $\mathcal{B}_i(\tv)$ and $s_0(\tv)$ is the affinity dimension. In particular, for every $\mathcal{A}\in V$
$$
\dim_H\mu^K=\dim_H\Lambda=\dim_B\Lambda=s_0(\mathcal{B})\text{ for $\mathcal{L}_N$-a.e $\mathcal{B}\in Q(\mathcal{A})$.}
$$
By Lemma~\ref{lfoli}, $Q$ is a measurable foliation of $V$, thus, by Rokhlin's Theorem
$$
\dim_H\mu^K=\dim_H\Lambda=\dim_B\Lambda=s_0(\mathcal{A})
\text{ for }\mathcal{L}_{4N}\text{-a.e. }\mathcal{A}\in V.
$$
Since $V$ was arbitrary, the statement follows.
\end{proof}

\end{document}